\documentclass[12pt]{amsart}
\usepackage{amsthm}
\usepackage{amsmath}
\usepackage{amssymb}
\usepackage{verbatim}
\usepackage{longtable,enumitem,color}
\usepackage{stmaryrd}
\usepackage{graphicx}
{\providecommand{\noopsort}[1]{} 
\usepackage{hyperref}
\usepackage[margin=1in]{geometry}

\usepackage{url}

\usepackage{chngcntr}
\counterwithin{table}{section}

\theoremstyle{plain}
\newtheorem{theorem}{Theorem}[section]\newtheorem*{nonumbertheorem}{Theorem}
\newtheorem{corollary}[theorem]{Corollary}\newtheorem*{nonumbercorollary}{Corollary}
\newtheorem{lemma}[theorem]{Lemma}

\theoremstyle{definition}

\theoremstyle{remark}
\newtheorem{remark}[theorem]{Remark}
\newtheorem{example}[theorem]{Example}
\numberwithin{equation}{section}

\newcommand{\ep}{\epsilon}

\newcommand{\Z}{\mathbb{Z}}\newcommand{\Q}{\mathbb{Q}}
\newcommand{\C}{\mathbb{C}}\newcommand{\HH}{\mathbb{H}}
\newcommand{\s}{\mathbb{S}}\newcommand{\Ca}{\mathrm{Ca}}\newcommand{\pp}{\mathrm{P}}

\newcommand{\Sp}{\mathrm{Sp}}

\newcommand{\SO}{\mathrm{SO}}

\newcommand{\SU}{\mathrm{SU}}
\newcommand{\PU}{\mathrm{PU}}
\newcommand{\Un}{\mathrm{U}}
\newcommand{\sone}{\mathrm{S}^1}
\newcommand{\Gtwo}{\mathrm{G}_2}

\DeclareMathOperator{\cod}{cod}

\DeclareMathOperator{\dk}{dk}

\newcommand{\of}[1]{\left(#1\right)}

\newcommand{\pfrac}[2]{\left(\frac{#1}{#2}\right)}
\newcommand{\inner}[1]{\left\langle #1 \right\rangle}
\newcommand{\st}{~|~}
\newcommand{\tensor}{\otimes}

\renewcommand{\subset}{\subseteq}

\newcommand{\cc}{{\mathbb{C}}}                                     
\newcommand{\hh}{{\mathbb{H}}}                                     
\newcommand{\qq}{{\mathbb{Q}}}                                     
\newcommand{\rr}{{\mathbb{R}}}                                     
\newcommand{\zz}{{\mathbb{Z}}}                                     
\newcommand{\dif} {{\operatorname{d}}}                             

\newcommand{\biq}[2]{#1\;\!\!\!\sslash \;\!\!\!#2}                 

\title[Positive curvature and torus symmetry, I]{Positive curvature and torus symmetry in small dimensions, I -- Dimensions 10, 12, 14, and 16}


\author{Manuel Amann and Lee Kennard}

\begin{document}

\begin{abstract}
This is the first part of a series of papers where we compute Euler characteristics, signatures, elliptic genera, and a number of other invariants of smooth manifolds that admit Riemannian metrics with positive sectional curvature and large torus symmetry. In the first part, the focus is on even-dimensional manifolds in dimensions up to $16$. Many of the calculations are sharp and they require less symmetry than previous classifications. When restricted to certain classes of manifolds that admit non-negative curvature, these results imply diffeomorphism classifications. Also studied is a closely related family of manifolds called positively elliptic manifolds, and we prove the Halperin conjecture in this context for dimensions up to $16$ or Euler characteristics up to $16$.
\end{abstract}

\maketitle

\thispagestyle{empty}

\section*{Introduction}

In this article, we prove topological obstructions to the existence of Riemannian metrics with positive sectional curvature and large symmetry. This is part of a well established research program was initiated by K.~Grove in the 1990s and has led to the proofs of a large number of topological obstructions, the construction of many interesting examples, and the development of new tools for studying Riemannian manifolds with curvature bounds and symmetry. For some recent surveys, we refer the reader to Grove \cite{Grove09}, Wilking \cite{Wilking07}, and Ziller \cite{Ziller07,Ziller14}.

In this article, we focus on torus symmetry, as it is perhaps the most understood. Three fundamental results in this area are due to Grove and Searle \cite{GroveSearle94}, Fang and Rong \cite{FangRong05}, and Wilking \cite{Wilking03}. They prove, respectively, equivariant diffeomorphism, homeomorphism, and homotopy and cohomology classifications for positively curved manifolds with torus symmetry, where the rank of the torus action is bounded from below by a constant that only depends on the dimension of the manifold.

While these results provide strong classifications that hold in arbitrary dimensions, they do not always reduce to the best known results in small dimensions. For example, in dimensions $2$ and $3$, positive sectional curvature by itself already implies that the manifold is diffeormorphic to a sphere. This follows from the Gauss--Bonnet theorem, the classification of surfaces, and Hamilton's work on the Ricci flow. Furthermore, in dimension $4$, the best result is due to Hsiang--Kleiner and Grove--Wilking (see \cite{HsiangKleiner89,GroveWilking14}, cf. \cite{Galaz-GarciaRadeschi15,GeRadeschi15,PaternainPetean03}), while in dimension $5$, the best result is due to Rong and Galaz-Garcia--Searle (see \cite{Rong02, Galaz-GarciaSearle14}, cf. \cite{Gozzi15,Simas16}). In dimensions $6$ and $7$, the above results are the best known, however there remains a large gap in our understanding due to the vast number of known positively curved examples (see \cite{Ziller07, Dearricott11,GroveVerdianiZiller11,PetersenWilhelm-pre}).

This article focuses on even dimensions $8$ through $16$. Our philosophy is motivated by Dessai \cite{Dessai11}, which considers positively curved manifolds with torus symmetry in dimension $8$. The result of Fang and Rong implies that a closed, simply connected, positively curved $8$--manifold with $T^3$ symmetry is homeomorphic to $\s^8$, $\C\pp^4$, or $\HH\pp^2$, i.e., to one of the manifolds known to admit a positively curved metric. Dessai studies precisely this problem, except that he only assumes $T^2$ symmetry, and he computes a number of topological invariants, including the Euler characteristic $\chi(M)$ and signature $\sigma(M)$. These computations provide obstructions to the existence of positively curved Riemannian metrics on $8$--manifolds with $T^2$ symmetry.

\begin{nonumbertheorem}[Dessai, 2011]
If $M^{8}$ is a closed, simply connected Riemannian manifold with positive sectional curvature and $T^2$ symmetry, then one of the following occurs:
	\begin{itemize}
	\item $\chi(M) = 2$, $\sigma(M) = 0$, and, if $M$ is spin, the elliptic genus vanishes.
	\item $\chi(M) = 3$, $\sigma(M) = \pm 1$, and, if $M$ is spin, the elliptic genus is constant.
	\item $\chi(M) = 5$, $\sigma(M) = \pm 1$, and $M$ is not spin.
	\end{itemize}
\end{nonumbertheorem}

Dessai also restricts his results to well studied classes of non-negatively curved manifolds (e.g., biquotients and certain cohomogeneity one manifolds), and for such manifolds his results imply stronger classification (e.g., up to diffeomorphism).

The main results of this article extend this work of Dessai into dimensions $10$, $12$, $14$, and $16$. For example, we prove the following (see Theorem \ref{thm:dim10} for a more detailed statement).

\begin{nonumbertheorem}[Theorem 5.1]\label{main:dim10}
If $M^{10}$ is a closed, simply connected Riemannian manifold with positive sectional curvature and $T^3$ symmetry, then one of the following occurs:
	\begin{itemize}
	\item $M$ is homeomorphic to $\s^{10}$.
	\item $\chi(M) = \chi(\C\pp^5)$ and $H_i(M;\Z) \cong H_i(\C\pp^5;\Z)$ for $i \leq 3$.
	\end{itemize}
\end{nonumbertheorem}

We remark that the sphere and complex projective space are the only manifolds known to admit positive curvature in dimension $10$. There is a similar story in dimension $14$.

\begin{nonumbertheorem}[Theorem 7.1]\label{main:dim14}
If $M^{14}$ is a closed, simply connected, positively curved Riemannian manifold with $T^4$ symmetry, then one of the following occurs:
	\begin{itemize}
	\item $\chi(M) = 2$ and $M$ is $3$--connected.
	\item $H_*(M;\Z) \cong H_*(\C\pp^7;\Z)$, and the cohomology is generated by some $z \in H^2(M;\Z)$ and $x \in H^4(M;\Z)$ subject to the relation $z^2 = m x$ for some integer $m$.
	\end{itemize}
\end{nonumbertheorem}

In dimensions divisible by four, the quaternionic projective spaces arise as examples. In addition, dimension $16$ is home to the Cayley plane. The sphere, the Cayley plane, $\HH\pp^4$, and $\C\pp^8$ admit positively curved metrics with $T^4$ symmetry. The next result shows that any closed, simply connected $16$--manifold with positive curvature and $T^4$ symmetry has the same Euler characteristic as one of these spaces ($2$, $3$, $5$, or $9$, respectively). It also provides a sharp calculation of the signature and elliptic genus under these assumptions.

\begin{nonumbertheorem}[Theorem 8.1]\label{main:dim16}
If $M^{16}$ is a closed, simply connected Riemannian manifold with positive sectional curvature and $T^4$ symmetry, then one of the following occurs:
	\begin{itemize}
	\item $\chi(M) = 2$ and $\sigma(M) = 0$.
	\item $\chi(M) = 3$, $\sigma(M) = \pm 1$, and $M$ is $2$--connected.
	\item $\chi(M) = 5$, $\sigma(M) = \pm 1$, and $H_{2+4i}(M;\Z) = 0$ for all $i$.
	\item $\chi(M) = 9$, $\sigma(M) = \pm 1$, $H_i(M;\Z) \cong H_i(\C\pp^8;\Z)$ for $i \leq 4$, and $M$ is not spin. 
	\end{itemize}
In any case, if $M$ is spin, then the elliptic genus is constant.
\end{nonumbertheorem}

We now discuss dimension $12$. The rank one symmetric spaces $\s^{12}$, $\C\pp^6$, and $\HH\pp^3$ admit positively curved metrics with $T^4$ symmetry, and they are the only spaces up to tangential homotopy that admit such metrics by Wilking \cite[Theorem 2]{Wilking03}. If one requires only $T^3$ symmetry, there is no classification, and there is one known additional example, the Wallach manifold $W^{12} = \Sp(3)/\Sp(1)^3$.

The Euler characteristics of these manifolds are $2$, $7$, $4$, and $6$, respectively, and the absolute values of their signatures are $0$ or $1$, according to the parity of $\chi(M)$. Moreover, the elliptic genus is constant for each of these manifolds, and $\C\pp^6$ is not spin while the other three are spin. The next theorem provides a partial recovery of all of these properties. To state it, we denote by $C(6)$ the maximum Euler characteristic achieved by a closed, simply connected, positively curved $6$--manifold with $T^2$ symmetry. Note that $C(6) < \infty$ by Gromov's Betti number estimate (see \cite{Gromov81}). In fact, it is not difficult to see that $C(6) \in \{6,8,\ldots,14\}$ (see Lemma \ref{lem:C6}).

\begin{nonumbertheorem}[Theorem 6.2]\label{main:dim12}
Let $M^{12}$ be a closed, simply connected, positively curved Riemannian manifold with $T^3$ symmetry. Either $\chi(M) \in \{2,4,6,\ldots,C(6)\}$ or $M$ is not spin and $7 \leq \chi(M) \leq \frac{7}{4} C(6)$. Moreover,
	\begin{enumerate}
	\item if $C(6) = 6$, then $\chi(M) \in\{2,4,6\}$ or $M$ is not spin and $\chi(M) \in \{7,8,9\}$.
	\item if $M$ is rationally elliptic, then $\chi(M) \in\{2,4,6,7,8,9,10,12\}$.
	\end{enumerate}
Regarding the signature and elliptic genus, the following hold:
	\begin{enumerate}
	\item If $\chi(M) \leq 13$, then $|\sigma(M)| \in \{0,1\}$ according to the parity of $\chi(M)$.
	\item If $M$ is spin, then the elliptic genus is constant.
	\end{enumerate}
\end{nonumbertheorem}

In light of the known examples in dimension $12$, this result is sharp in the spin case under the additional assumption that $C(6) = 6$, which coincides with the Euler characteristics of the Wallach manifold $\SU(3)/T^3$ and the Eschenburg biquotient $\SU(3)/\!/T^3$. No compact $6$--manifold with larger Euler characteristic is known to admit a positively curved metric with $T^2$ symmetry.

In each theorem above, if the torus $T^r$ is replaced by $T^{r+1}$, there are stronger classifications due to Fang--Rong \cite{FangRong05} and Wilking \cite{Wilking03}. In particular, the Cayley plane and the Wallach $12$--manifold are excluded. If instead the $T^r$ is replaced by $T^{r-1}$, then little is known beyond the positivity of the Euler characteristic and the vanishing of some higher $\hat A$ genera (see Rong--Su \cite{RongSu05} and Dessai \cite{Dessai05, Dessai07}).

For the proofs of these results, we build upon on a great deal of previous work. 
One major source of ideas is Wilking \cite{Wilking03}, where the following are developed:
	\begin{enumerate}
	\item the connectedness lemma and the resulting periodicity in cohomology,
	\item  techniques for studying the fixed-point sets of involutions,
	\item spherical recognition theorems, and d
	\item the classification of maximal rank, smooth torus actions on $\HH\pp^n$.
	\end{enumerate}
In the first three of these cases, refinements have been made or consequences have been deduced. In particular, we rely on the refinements of (1) and (2) in our previous work (see \cite{Kennard13,AmannKennard14,AK3}). In this article, we add to this a simple but useful result related to (2) which we call the \textit{containment lemma} (see Lemma \ref{lem:Containment}). Also useful for our purposes is Lemma \ref{lem:cod4dk1Qsphere}, which is proven here using (3).

A second important source of ideas and motivation is Grove--Searle \cite{GroveSearle94}. First, we apply in a crucial way the \textit{equivariant} diffeomorphism classification of positively curved manifolds with maximal symmetry rank. We apply this to prove Lemma \ref{lem:GroveSearleTrick}, which is crucial to the proof of our result in dimension $12$. In addition, we prove a partial generalization of their classification of fixed-point homogeneous circle actions, which we call the \textit{codimension two lemma} (see Lemma \ref{lem:cod2}).

A third important source of ideas, and the main motivation for this article, is Dessai \cite{Dessai11}. There is a fair amount of work required to understand the \textit{global} picture of the fixed-point set data. Our basic strategy is the same: When there are not enough fixed-point sets of small codimension to classify the homotopy type using the connectedness lemma, one obtains \textit{isotropy rigidity} at fixed-point sets of the torus action. For the reader hoping to get just a passing idea of how this strategy works, we recommend the proof of the dimension $14$ case, as it involves the smallest number of special cases and other hiccups. For the more interested reader, the proof of the dimension $12$ case is by far the most involved but also, we believe, the most interesting. In particular, there is a significant amount of combinatorial analysis required in this dimension that comes out of the isotropy rigidity.

Fourth, regarding the elliptic genus calculations, see Section \ref{sec:EllipticGenus}. These calculations are motivated by a question of Dessai \cite{Dessai05, Dessai07} and work of Weisskopf \cite{Weisskopf}. 

\smallskip

All manifolds studied in this article are shown to have positive Euler characteristic. In addition, they are positively curved, so a conjecture of Bott--Grove--Halperin suggests that they are rationally elliptic (see Grove \cite{Grove02}). Putting these properties together, it is suspected that the class of manifolds studied in this paper is closely related to the class of \textit{positively rationally elliptic}, or $F_0$, spaces. Motivated by this, we compute the homotopy groups of $F_0$ spaces of formal dimension at most $16$ (see Section \ref{sec:F0spaces}). The resulting tables are used to study the four questions we discuss next.

First, we specialize the results above to rationally elliptic spaces and derive rational homotopy classifications in dimensions $10$, $14$, and $16$. We also provide a partial classification of this kind in dimension $12$. For this, the Euler characteristic calculations are helpful but not sufficient. One must also apply the conclusions above regarding the product structure in cohomology. See Section \ref{sec:Elliptic}.

Second, we specialize further to biquotients, a large class of manifolds that contains all homogeneous spaces and that provides a source of numerous examples of manifolds admitting positive sectional curvature, as well as weaker notions such as positive curvature almost everywhere (see, for example, \cite{Ziller07}, DeVito \cite{DeVito14,DeVitoEtAl14,DeVito-pre}, Kerin \cite{Kerin11, Kerin12}, Kerr--Tapp \cite{KerrTapp14}, and Wilking \cite{Wilking02}). In our context, our results in dimensions $10$, $14$, and $16$ imply diffeomorphism classifications when restricted to the case of biquotients.

\begin{nonumbercorollary}[Theorem 13.1]
Let $M^n$ be a closed, simply connected biquotient that independently admits a positively curved Riemannian metric with $T^r$ symmetry.
\begin{itemize}
\item If $n= 10$ and $r \geq 3$, then $M$ is diffeomorphic to $\s^{10}$, $\C\pp^5$, $\s^2 \tilde{\times}\HH\pp^2$, $\SO(7)/(\SO(5)\times \SO(2))$, or $\Delta \SO(2)\backslash \SO(7)/\SO(5)$.
\item If $n= 14$ and $r \geq 4$, then $M$ is diffeomorphic to $\s^{14}$, $\C\pp^7$, $\s^2 \tilde{\times}\HH\pp^3$, $\SO(9)/(\SO(7)\times \SO(2))$, or $\Delta \SO(2)\backslash \SO(9)/\SO(7)$.
\item If $n = 16$ and $r \geq 4$, then $M$ is diffeomorphic to $\s^{16}$, $\C\pp^8$, $\HH\pp^4$, or $\Ca\pp^2$.
\end{itemize}
Here $\s^2 \tilde \times \HH\pp^m$ denotes one of the two diffeomorphism types of total spaces of $\HH\pp^m$--bundles over $\s^2$ whose structure group reduces to circle acting linearly.
\end{nonumbercorollary}

This corollary is proved in two steps. First, we apply the rational ellipticity of biquotients  to derive the rational homotopy type. We then apply a diffeomorphism classification of Kapovitch and Ziller \cite{KapovitchZiller04}, together with  a recent generalization due to DeVito \cite{DeVito-pre2}, to conclude the result, up to a small number of other possibilities. To exclude these, we return to what our calculations imply about the integral cohomology, and this is sufficient to complete the classification.

We remark that the $16$-dimensional case of this corollary provides a diffeomorphism characterization of the Cayley plane among biquotients admitting positively  curved metrics with $T^4$ symmetry. To our knowledge, previous results along these lines have either had too strong a symmetry assumption to allow the Cayley plane or too weak an assumption to detect it.

Third, we specialize even further in dimension $12$ to the class of symmetric spaces, and here we obtain a diffeomorphism classification (see Theorem \ref{thm:dim12symmetricspaces}).

Finally, we study the conjecture of Halperin that any fibration $E \to B$ of simply connected spaces where the fiber $F$ is an $F_0$ space has the property that $H^*(E;\Q) \cong H^*(B;\Q) \tensor H^*(F;\Q)$ as $H^*(B;\Q)$--modules.

\begin{nonumbertheorem}[Theorem 11.6]
The Halperin conjecture holds for $F_0$ spaces of formal dimension at most $16$ or Euler characteristic at most $16$.
\end{nonumbertheorem}

For further discussion of $F_0$ spaces and Halperin's conjecture, cases in which this conjecture is known to hold, and the proof of this theorem, see Sections \ref{sec:F0spaces} and \ref{sec:Halperin}.

\subsection*{Acknowledgements}
The second author was supported by National Science Foundation Grant DMS-1404670/DMS-1622541. The authors would also like to acknowledge the support of NSF Grant DMS-1440140 while they were in residence at the MSRI in Spring 2016.

\tableofcontents

\section{Preliminaries}\label{sec:Preliminaries}\smallskip

There are a large number of relatively old results from the theory of transformation groups and newer tools developed over the past two decades that have grown out of work on the Grove program. We attempt to efficiently summarize those which we use.

Let $M$ be an even-dimensional, closed Riemannian manifold with positive sectional curvature, and let $T$ denote a torus that acts isometrically on $M$. A theorem of Berger states that the fixed-point set $M^T = \{x \in M \st g(x) = x ~\mathrm{for~all}~g\in T\}$ of the torus is non-empty. In general, for $g \in T$, each component $N$ of the fixed-point set $M^g = \{x \in M \st g(x) = x\}$ is a closed, even-dimensional, totally geodesic (hence positively curved) embedded submanifold on which $T$ acts, possibly ineffectively. Applying Berger's theorem to the induced $T$--action on $N$, we see that every fixed-point component of every isometry in $T$ contains a fixed point of $T$. We use this fact frequently.

Another issue is proving that components $N \subseteq M^g$ are orientable. If $N$ is a fixed-point component of some subgroup $H\subseteq T$ not equal to $\Z_2$, then $N$ inherits orientability from $M$. Also, if $N$ has dimension greater than $\frac{1}{2} \dim M$, then Wilking's connectedness lemma (see below) implies that $N$ is simply connected. In some cases, other arguments are used. For example, if $N \subseteq M^{\Z_2}$ is a component and $\Z_2 \subseteq \sone \subseteq T$, and if $Q \subseteq N$ is a fixed-point component of $\sone$ inside $N$ such that $\dim(Q) \geq \frac 1 2 \dim(N)$, then $Q$ is orientable, and hence simply connected, and so $N$ is simply connected, and hence orientable, by the connectedness lemma.

Next, we recall some results from Smith theory. The first is due to Conner and Kobayashi (see \cite{Conner57,Kobayashi58}): The Euler characteristic of $M$ and its fixed-point set $M^T$ coincide. The same holds for the signature, i.e., $\sigma(M) = \sigma(M^T)$ (see \cite[Theorem 2.4]{Dessai11}). For this latter result, one has to take care how one assigns orientations to the components of $M^T$, however for our purposes this will not matter since we will always show $|\sigma(M)| \leq 1$ by showing that $|\sigma(M^T)| \leq \sum |\sigma(F)| \leq 1$, where the sum runs over components $F\subseteq M^T$. Another result due to Conner is that sum of the odd Betti numbers of $M^T$ is at most that of $M$, and likewise for the even Betti numbers.

Regarding the Euler characteristic, we use frequently the following inclusion-exclusion property. If $M^T$ is contained in the union $N_1 \cup N_2$ where $N_1$ and $N_2$ admit induced $T$--actions, then $M^T = (N_1 \cup N_2)^T$. Applying the property above together with the Mayer--Vietoris sequence, we conclude
	\[\chi(M) = \chi(N_1) + \chi(N_2) - \chi(N_1 \cap N_2).\]
There are obvious extensions of this formula in the case of three or more submanifolds.

\smallskip

To close this section on preliminaries, we discuss a collection of results related to Wilking's connectedness lemma. We recall these at many points in the paper, so for ease of reference within this paper, we give them names.

\begin{theorem}[Frankel, \cite{Frankel61}]
Let $M^n$ be a closed Riemannian manifold with positive sectional curvature. If $N_1, N_2 \subseteq M$ are closed, totally geodesic, embedded submanifolds such that $\cod(N_1) + \cod(N_2) \leq n$, then $N_1$ and $N_2$ non-trivially intersect.
\end{theorem}

Throughout this paper, $\cod(N)$ denotes the codimension of a submanifold $N\subseteq M$.

\begin{theorem}[Connectedness lemma, \cite{Wilking03}]
Let $M^n$ be a closed Riemannian manifold with positive sectional curvature.
	\begin{enumerate}
	\item If $N^{n-k} \to M$ is a closed, totally geodesic, embedded submanifold, then the inclusion is $(n-2k+1)$--connected.
	\item If $N^{n-k} \to M$ is as above, and if $N$ is a fixed point component of an isometric action by a Lie group $G$, then the inclusion is $(n-2k+1+\delta)$--connected, where $\delta$ is the dimension of the principal orbits of $G$.
	\item If $N_i^{n-k_i} \to M$ are closed, totally geodesic, embedded submanifolds with $k_1 \leq k_2$, then the inclusion $N_1\cap N_2 \to N_2$ is $(n-k_1-k_2)$--connected.
	\end{enumerate}
\end{theorem}

Note that, for our purposes, if $M$ is a closed, positively curved manifold and if $H$ acts isometrically on $M$, then every component of $M^H$ is a closed, embedded, totally geodesic submanifold.

When the codimensions in the connectedness lemma are sufficiently small, there are strong cohomological consequences. The following is a corollary of the connectedness lemma together with \cite[Lemma 2.2]{Wilking03}.

\begin{corollary}[Wilking's periodicity corollary]\label{cor:PeriodicityCorollary}
Let $M^n$ be a closed, simply connected, positively curved Riemannian manifold.
	\begin{enumerate}
	\item If $N^{n-k} \to M$ is a closed, totally geodesic, embedded submanifold, it follows that  $H^{k-1 \leq * \leq n - k + 1}(M;\Z)$ is $k$--periodic.
	\item If $N^{n-k} \to M$ is as above, and if $N$ is a fixed point component of an isometric action by a Lie group $G$, then $H^{k-1-\delta \leq * \leq n - k + 1 + \delta}(M;\Z)$ is $k$--periodic, where $\delta$ is the dimension of the principal orbits of $G$.
	\item If $N_i^{n-k_i} \to M$ are closed, totally geodesic, embedded submanifolds with $k_1 \leq k_2$, and if the intersection $N_1 \cap N_2$ is transverse, then $H^*(N_2;\Z)$ is $k$--periodic.
	\end{enumerate}
\end{corollary}

Throughout this paper, we say that $H^{m \leq * \leq n-m}(M;\Z)$ is $k$--periodic if there exists $x \in H^k(M;\Z)$ such that the multiplication map $H^i(M;\Z) \to H^{i+k}(M;\Z)$ induced by $x$ is a surjection from $H^m(M;\Z)$, an injection into $H^{n-m}(M;\Z)$, and an isomorphism everywhere in between. When $m = 0$, we say simply that $H^*(M;\Z)$ is $k$--periodic.

In \cite{Kennard13}, refinements to this corollary are proved by applying Steenrod powers to $k$--periodic cohomology rings. For this paper, we only require the following result:

\begin{lemma}[Classification of $4$--periodic cohomology]\label{lem:4periodicZ2cohomology}
Let $M^n$ be a simply connected closed manifold of dimension $n \geq 8$ such that $H^*(M;\Z)$ is $4$--periodic.
	\begin{enumerate}
	\item If $n \equiv 0\bmod{4}$, then $M$ is a cohomology $\s^n$, $\C\pp^{\frac n 2}$, or $\HH\pp^{\frac n 4}$.
	\item If $n \equiv 2\bmod{4}$, then either $M$ is a mod $2$ cohomology sphere or $M$ is a homology $\C\pp^{\frac{n}{2}}$ with cohomology generated by some $z \in H^2(M;\Z)$ and $x\in H^4(M;\Z)$ such that $z^2 = m x$ for some $m\in \Z$.
	\end{enumerate}
\end{lemma}

\begin{proof}
If $n \equiv 0 \bmod 4$, this is consequence of the definition and Poincar\'e duality. For $n \equiv 2 \bmod{4}$, the proof requires an additional argument that shows  $M$ is a mod $2$ cohomology $\s^n$, $\C\pp^{\frac{n}{2}}$, or $\s^2 \times \HH\pp^{\frac{n-2}{4}}$ (see \cite[Section 6]{Kennard13}). In particular, $H^3(M;\Z_2) = H^5(M;\Z_2) = 0$, and so $H^4(M;\Z)$ is isomorphic to $\Z_{2k+1}$ for some integer $k$ or to $\Z$ by the universal coefficients theorem and the definition of periodicity. Applying the definition of periodicity and Poincar\'e duality, these two cases correspond to and imply the two possibilities claimed in the lemma.
\end{proof}

\smallskip\section{Containment lemma}\label{sec:Containment}\smallskip

The first result applies Frankel's theorem to provide a sufficient condition for the fixed-point set $M^T$ of the torus action to be contained in a small number of fixed-point components of involutions. We will use this lemma many times, so we will call it the {\it containment lemma}.

\begin{lemma}[Containment lemma]\label{lem:Containment}
Suppose $M^n$ is a closed, positively curved Riemannian manifold, and assume $T$ is a torus acting isometrically on $M$. Fix $x \in M^T$. Let $H \subseteq T$ denote a subgroup isomorphic to $\Z_2^r$. Let $\delta = 4$ if $n \equiv 0 \bmod 4$ and $M$ is spin, and otherwise let $\delta = 2$ if $n$ is even and $\delta = 1$ if $n$ is odd. If
	\[\cod\of{M^H_x} + \cod\of{M^H_y} < \pfrac{2^r-1}{2^{r-1}}(n+\delta),\]
then $y \in \bigcup_{\iota \in H \setminus\{1\}} M^\iota_x$.
\end{lemma}

Note that every component of $M^H$ has codimension at most $n$. Hence, the containment lemma implies that, if
	\[\cod\of{M^H_x} < n + 2\delta - \frac{n+\delta}{2^{r-1}},\]
then $M^T \subseteq \bigcup M^\iota_x$, where the union runs over non-trivial elements $\iota \in H$. For example, when $M$ has even dimension, this assumption is satisfied in each of the following two cases:
	\begin{itemize}
	\item $H \cong \Z_2^2$ and $\cod\of{M^H_x} \leq \frac{n}{2} + 2$.
	\item $H \cong \Z_2^3$ and $\cod\of{M^H_x} \leq \frac{3n}{4} + 3$.
	\end{itemize}
We will use this consequence frequently. We remark that, if we take $r = 1$ in this statement, this is just a restatement of Frankel's theorem for the case of components of the fixed-point set of some copy of $\Z_2$ in $T$.

\begin{proof}
Let $y \in M^T$ and $\iota \in H$. If $y \not\in M^\iota_x$ then the components $M^\iota_x$ and $M^\iota_y$ are disjoint, so Frankel's theorem implies
	\[\cod\of{M^\iota_y} + \cod\of{M^\iota_x} > n.\]
If, moreover, $M$ is spin and $n \equiv 0 \bmod{4}$, the sum of these codimensions is divisible by four and greater than $n$, hence it is at least $n + 4$. In any case, the left-hand side is at least $n + \delta$, where $\delta$ is defined as in the theorem. The lemma follows by summing these inequalities over $\iota \in H$ and using the fact that $\sum_{\iota \in H} \cod{M^\iota_x} = 2^{r-1} \cod\of{M^H_x}$ since $H$ is some copy of $\Z_2^r$ inside $T$ (see Borel \cite{Borel60}).
\end{proof}

\smallskip\section{Recognition theorems for spheres}\label{sec:Recognition}\smallskip

It is a basic result from Smith theory that a smooth action by $\Z_p$ on a mod $p$ homology sphere has fixed-point set a mod $p$ homology sphere (see Bredon \cite[Chapter III, Theorem 5.1]{Bredon72}). Wilking proved two results which can be viewed as partial converses to this theorem. The first provides a sufficient condition for recognizing when a manifold is homeomorphic to a sphere (see \cite[Theorem 4.1]{Wilking03}).

\begin{theorem}[Recognition theorem, even-dimensional case]\label{thm:Wilking4.1}
Let $M^{2m}$ be a compact $(m-k)$--connected manifold. Suppose that $T^{2k-1}$ acts smoothly and effectively on $M$ with non-empty fixed-point set. Assume that every $\sigma \in T^{2k-1}$ of prime order $p$ has the property that its fixed-point set is either empty or a mod $p$ homology sphere. Then $M$ is homeomorphic to a sphere.
\end{theorem}

The second provides a partial recognition result for mod $p$ homology spheres (see \cite[Theorem 5.1]{Wilking03}):

\begin{theorem}\label{thm:Wilking5.1}
Let $M^n$ be a closed, simply connected, positively curved Riemannian manifold, and let $p$ be a prime. If $\Z_p$ acts isometrically on $M$ with connected fixed-point set of codimension $k$, then $H^i(M;\Z_p) = 0$ for $k \leq i \leq n - 2k + 1$. If, moreover, $M^{\Z_p}$ is fixed by a circle acting isometrically on $M$, then $H^{k-1}(M;\Z_p) = 0$ and $H^{n - 2k + 2}(M;\Z_p) = 0$ as well.
\end{theorem}

\begin{proof}
The first statement is just \cite[Theorem 5.1]{Wilking03}, which is based on the fact that $M^{\Z_p} \to M$ is $(n-2k+1)$--connected according to the connecteness lemma. When $M^{\Z_p}$ is fixed by a circle as in the second assertion, the connectedness lemma implies that the inclusion $M^{\Z_p} \to M$ is $(n-2k+2)$--connected, and the proof of \cite[Theorem 5.1]{Wilking03} imply that $H^{k-1}(M;\Z_p) = 0$ and $H^{n-2k+2}(M;\Z_p) = 0$ as well.
\end{proof}

Using Theorem \ref{thm:Wilking5.1}, we prove the following, which can be viewed as another recognition theorem for the sphere.

\begin{lemma}\label{lem:cod4dk1Qsphere}
Let $M^n$ be a closed, simply connected, positively curved Riemannian manifold with $n = 2m \geq 10$. Assume there exists a circle $\sone$ acting isometrically on $M$ such that $M^{\sone}$ has a component $N$ of codimension four. If $b_2(M) = 0$ and $\chi(M) = \chi(N)$, then $M$ is homeomorphic to a sphere.
\end{lemma}

Note, in particular, that $b_2(M) = 0$ and $\chi(M) = \chi(N) = 2$ if $M$ is a rational homology sphere.

\begin{proof}
Let $\Z_p \subset \sone$. If $M^{\Z_p}$ has codimension two, then the result follows from Wilking's periodicity corollary. Suppose that $M^{\Z_p}$ has codimension four. By Frankel's theorem, $N$ is the only component of codimension four and all others are points or positively curved $2$--manifolds. Since $\chi(M) = \chi\of{M^{\Z_p}}$, it follows that $M^{\Z_p}$ is connected. By Theorem \ref{thm:Wilking5.1} it follows that $H^i(M;\Z_p) = 0$ for all $3 \leq i \leq n - 6$. For $n \geq 12$, it follows immediately from Poincar\'e duality and the assumption that $b_2(M) = 0$ that $M$ is an integral homology sphere. For $n = 10$, the same holds by a similar argument using the additional observation that $b_5(M) = \chi(M) - 2 = 0$. The lemma follows by the resolution of the Poincar\'e conjecture.
\end{proof}

\smallskip\section{Low codimension lemmas}\label{sec:LowCodim}\smallskip

It follows immediately from Wilking's periodicity corollary that a positively curved, oriented, closed manifold of odd dimension contains a codimension two, totally geodesic submanifold only if it is homeomorphic to $\s^n$. It is an open question whether an analogous result holds in the even dimensional case. By Grove and Searle's result, if the codimension two submanifold is fixed by an isometric circle action, then $M$ is in fact diffeomorphic to $\s^n$ or $\C\pp^{n/2}$. For our purposes, we will use the following, different partial result in this direction.

\begin{lemma}[Codimension two lemma]\label{lem:cod2}
Let $M^n$ be a closed, simply connected, positively curved Riemannian manifold with $T^2$ symmetry. If some involution $\iota \in T^2$ has fixed-point set of codimension two, then $M$ is homeomorphic to $\s^n$ or $\C\pp^{n/2}$.
\end{lemma}

\begin{proof}
Let $N \subseteq M^\iota$ be a component of codimension two. By Wilking's periodicity corollary, it follows that $M$ is an cohomology $\s^n$ or $\C\pp^{n/2}$ if $b_2(M) \leq 1$. In the first of these cases, the homeorphism classification follows by Perelman's resolution of the Poincar\'e conjecture. In the second case, the homeomorphism classification follows from Lemma 3.6 in Fang--Rong \cite{FangRong05}. Hence it suffices to show that $b_2(M) \leq 1$. We may assume that $n \geq 8$, since otherwise this holds by the theorem of Hsiang and Kleiner and the connectedness lemma.

Suppose first that an involution $\iota'\in T^2 \setminus\langle\iota\rangle$ exists whose fixed-point set contains a component $N'$ has codimension at most $\frac{n-1}{2}$. The intersection is transverse, so Wilking's periodicity corollary implies that $N'$ is $2$--periodic. By the connectedness lemma, $b_2(M) \leq b_2(N') \leq 1$, as required.

Suppose now that $\cod(M^{\iota'}) \geq \frac n 2$ for all involutions $\iota' \in T^2\setminus\langle\iota\rangle$. Consider the isotropy representation at any point $x \in M^T \setminus N$. The codimensions of $M^{\iota'}_x$ and $M^{\iota\iota'}_x$ sum to $n$ and hence equal $\frac{n}{2}$.  Fixing any such involution $\iota'$ and denoting its component of maximal dimension by $N'$, it follows by Frankel's theorem that $M^{T^2} \subseteq N \cup N'$. Applying the inclusion-exclusion property to this containment, we have
	\[\chi(M) - \chi(N) = \chi(N') - \chi(N \cap N').\]
By Wilking's periodicity corollary, Parts (1) and (3), the left-hand side equals $b_2(M)$, the right-hand side equals $b_2(N')$, and $b_2(N') \leq 1$. This completes the proof.
\end{proof}

Using the codimension two lemma, we prove a similar lemma for the case of codimension four fixed-point sets.

\begin{lemma}[Codimension four lemma, part 1]\label{lem:cod4part1}
Let $M^n$ be a closed, simply connected, positively curved Riemannian manifold with even dimension $n \geq 12$ and $T^r$ symmetry with $r \geq 3$. Assume there exists $x \in M$ and involutions $\iota_1,\iota_2 \in T^r$ such that $\cod\of{M^{\iota_1}_x} = 4$ and $\cod\of{M^{\iota_2}_x} < \frac{n}{2}$. One of the following occurs:
	\begin{enumerate}
	\item $M$ is a cohomology $\s^n$, $\C\pp^{\frac n 2}$, or $\HH\pp^{\frac n 4}$, or $M$ is a homology $\C\pp^{2m+1}$ with cohomology generated by some $z \in H^2(M;\Z)$ and $x \in H^4(M;\Z)$ such that $z^2 = m x$ for some $m \in \Z$.
	\item $n = 4m+2$, $M \sim_{\Z_2} \s^n$, and $T^r$ acts almost effectively on $M^{\iota_1}_x$.
	\item $n = 4m+2$, $\cod(M^{\iota_2}_x) = \frac{n}{2} - 1$ and $T^r$ acts almost effectively on $M^{\iota_1}_x$ or $M^{\iota_2}_x$.
	\end{enumerate}
\end{lemma}

\begin{proof}
Let $N_1 \subseteq M^{\iota_1}$ denote the component of codimension four, and let $N_2 \subseteq M^{\iota_2}$ denote the component of codimension $k_2 \leq \frac n 2 - 1$. Also let $N_{12} \subseteq M^{\iota_1\iota_2}$ be the component containing $N_1 \cap N_2$, which is connected by the connectedness lemma. We may assume that $N_2$ has maximal dimension among all such choices, so it admits $T^2$ symmetry. As a result, $N_2$ has $4$--periodic cohomology by the connectedness lemma (if $N_1 \cap N_2$ is transverse) or the codimension two lemma (if not).

First assume $k_2 \leq \frac{n}{2} - 2$. We claim that $b_2(M) = b_6(M)$. Since $N_2$ has $4$--periodic cohomology, $H^3(M;\Z_2) = 0$ and $H^5(N_2;\Z) = 0$ by Lemma \ref{lem:4periodicZ2cohomology}. Since the inclusion $N_2 \to M$ is $5$--connected, $H^3(M;\Z_2) = 0$ and $H^5(M;\Z) = 0$ as well. By the periodicity corollary, $H^{2i+1}(M;\Z_2) = 0$ for all $i$ and
	\begin{equation}\label{eqn:codim4eqn1}
	\chi(M) - \chi(N_1) = b_4(M) + b_6(M).
	\end{equation}
By the connectedness lemma, $N_2$, $N_{12}$, and the common $2$--fold intersection $M^{\inner{\iota_1,\iota_2}}_x$ of any three of submanifolds $N_1$, $N_2$, and $N_{12}$ are $4$--periodic with common values of $b_2$ and $b_4$. Moreover, the $M^{T^r}$ is contained in $N_1 \cup N_2 \cup N_{12}$ by the containment lemma, so the inclusion-exclusion property of Euler characteristics implies
	\begin{equation}\label{eqn:codim4eqn2}
	\chi(M) - \chi(N_1) = b_2(M) + b_4(M).
	\end{equation}
Indeed, if $N_1 \cap N_2$ is transverse, then $N_{12} = N_1 \cap N_2$  by the connectedness lemma and $\chi(N_2) - \chi(N_1 \cap N_2) = b_2(N_2) + b_4(N_2)$ by $4$--periodicity. Similarly, if $N_1 \cap N_2$ is not transverse, then $N_2$, $N_{12}$, and $N_1 \cap N_2$ are actually $2$--periodic and the right-hand side equals $b_2(N_2) + b_2(N_{12}) = b_2(N_2) + b_4(N_{12})$. In both cases, the equality holds by the connectedness lemma. Comparing Equations \ref{eqn:codim4eqn1} and \ref{eqn:codim4eqn2}, we conclude the claim that $b_2(M) = b_6(M)$. Given this claim, it follows from Lemma \ref{lem:4periodicZ2cohomology}, the connectedness lemma, the naturality of cup products, and the periodicity corollary that Conclusion (1) holds or that $M$ is a mod $2$ homology sphere. Moreover, we may refine the latter conclusion as follows:
	\begin{itemize}
	\item If some circle in $T^r$ fixes $N_1$, then $M$ is homeomorphic to $\s^n$ by Lemma \ref{lem:cod4dk1Qsphere} and Conclusion (1) holds.
	\item If $n \equiv 0 \bmod{4}$, then Poincar\'e duality and the periodicity corollary implies that $H_3(M;\Z) \cong H^5(M;\Z)$, which vanishes since it injects into $H^5(N_2;\Z) = 0$. Hence $H^4(M;\Z)$ is torsion-free in this case, $N_2$ and hence $M$ are integral homology spheres, and Conclusion (1) holds.
	\item If neither of these cases occurs, then $T^r$ acts almost effectively on $N_1$, the dimension $n \equiv 2 \bmod{4}$, and Conclusion (2) holds.
	\end{itemize}

Assume from now on that $k_2 = \frac{n}{2} - 1$. Note that $n \equiv 2 \bmod{4}$ since $k_2$ is even. Moreover, assume that both $\dk(N_i) \geq 1$ since otherwise Conclusion (3) holds. In particular, $N_2$ has $4$--periodic cohomology and the inclusion $N_2 \to M$ is $4$--connected. We may assume further that $b_2(M) \geq 1$. Indeed, if $b_2(M) = 0$, then the periodicity corollary implies $b_4(M) = 0$ and hence that $M$ is a rational homology sphere. By Lemma \ref{lem:cod4dk1Qsphere}, $M$ is homeomorphic to $\s^n$ and Conclusion (1) holds.

Lemma \ref{lem:4periodicZ2cohomology} implies that $N_2$ is a cohomology $\C\pp^{\frac{n+2}{4}}$ or that $n \equiv 2 \bmod{8}$ and $H^*(N_2;\Z)$ is generated by some $z \in H^2(N_2;\Z)$ and $x \in H^4(N_2;\Z)$ such that $z^2 = m x$ for some $m \in \Z$. Since $N_2 \to M$ is $4$--connected, it follows by Poincar\'e duality that $M$ is a cohomology $\C\pp^{\frac n 2}$, so we may assume the cohomology of $N_2$ is as in the latter case. We have that $H^2(M;\Z) = \Z$, $H^3(M;\Z) = 0$, and $H_4(M;\Z)$ is the image of $H_4(N_2;\Z)$ under the map induced by inclusion. In particular, the induced injection $H^4(M;\Z) \to H^4(N_2;\Z)$ is either an isomorphism or the zero map. If it is an isomorphism, the naturality of cup products implies that $H^*(M;\Z)$ is generated by elements in degree two and four as is the case for $N_2$ and Conclusion (1) holds. If, instead, it is the zero map, then $H^4(M;\Z) = 0$ and hence $M \sim_\Z \s^2 \times \s^{n-2}$ and $N_1 \sim_\Z \s^2 \times \s^{n-6}$ by the periodicity corollary and the connectedness lemma. In particular, $\chi(M) = \chi(N_1)$, so $M^{\iota_1}$ is connected by Frankel's theorem. By the connectedness lemma again, the action of $\iota_1$ on $N_2$ has connected fixed point set $(N_2)^{\iota_1} = M^{\iota_1} \cap N_2 = N_1 \cap N_2$. If $\dim(N_2) \geq 12$, this implies $H^4(N_2;\Z_2) = 0$ by Lemma \ref{thm:Wilking5.1}, a contradiction, so we may assume $\dim(N_2) \leq 10$, which is equivalent to $n \leq 18$. Since $n \equiv 2 \bmod{8}$ and $n \geq 12$ by assumption, we have that $n = 18$ and $\dim(N_2) = 10$. Note that $\chi(M) = 4$ and $\chi(N_2) = 6$, so there exists a component $M^{\iota_2}_y \subseteq M^{\iota_2}$ with negative Euler characteristic. By Frankel's theorem, $\dim(M^{\iota_2}_y) \leq 6$. But $\chi(M^{\iota_2}_y) > 0$ if $M^{\iota_2}_y$ has dimension $0$, $2$, or $4$, or if $M^{\iota_2}_y$ admits an isometric circle action, so we may assume $\dim(M^{\iota_2}_y) = 6$ and that $T^r$ fixes $M^{\iota_2}_y$. Since $r \geq 3$, there exists an involution $\iota \in T^r \setminus\inner{\iota_1,\iota_2}$ such that $M^{\iota}_y$ strictly contains $M^{\iota_2}_y$. By replacing $M^{\iota}_y$ by $M^{\iota_2\iota}_y$, if necessary, we may assume $\cod(M^{\iota}_y) \leq \frac{1}{2} \cod(M^{\iota_2}_y) = 6$. But now $M^{\iota_1}_x$ and $M^{\iota}_y$ intersect by Frankel, and we may proceed as in the case where $\cod(M^{\iota_2}) \leq \frac{n}{2} - 2$. Altogether one concludes that this case does not actually occur.
\end{proof}

In the case where there is no involution $\iota_2$ as in Lemma \ref{lem:cod4part1}, we need a companion lemma.

\begin{lemma}[Codimension four lemma, part 2]\label{lem:cod4part2}
Let $M^n$ be a closed, simply connected, positively curved Riemannian manifold with even dimension and $T^r$ symmetry such that $r \geq 3$. Assume there exists an involution $\iota_1 \in T^r$ such that $\cod\of{M^{\iota_1}} = 4$ and that every other involution $\iota$ satisfies $\cod\of{M^\iota} \geq \frac{n}{2}$. Either
	\begin{enumerate}
	\item $M^{\iota_1}$ is connected, or
	\item $2^r$ divides $n$, and every involution $\iota \in T^r \setminus\inner{\iota_1}$ has $\cod\of{M^{\iota}} = \frac n 2$.
	\end{enumerate}
\end{lemma}

\begin{proof}[Proof of codimension four lemma, part 2]
Suppose that $M^{\iota_1}$ is not connected. Using Berger's theorem, choose $y \in (M^{\iota_1})^T$ such that $y$ does not lie in the component of $M^{\iota_1}$ of codimension four. Frankel's theorem implies $\cod\of{M^{\iota_1}_y} \geq n - 2$. Under the map $\Z_2^r \to \Z_2^{n/2}$ induced by the isotropy representation of $T^r$ at $y$,
	\[\iota_1 \mapsto (1, 1, 1, \ldots, 1, *),\]
where the asterisk is $0$ or $1$ according to whether $\cod\of{M^{\iota}_y}$ is $n-2$ or $n$. In the first case, since $r \geq 3$, we may choose $\iota_2 \in \Z_2^r \setminus\inner{\iota_1}$ whose image has a $0$ in the last entry. It then follows that $M^{\iota_2}_y$ or $M^{\iota_1\iota_2}_y$ has codimension less than $\frac n 2$, a contradiction. In the second case, we have that
	\[\cod\of{M^\iota_y} + \cod\of{M^{\iota\iota_1}_y} = n\]
for all $\iota \in \Z_2^r \setminus \inner{\iota_1}$. Under the assumption that $\cod\of{M^\iota_y} \geq \frac n 2$ for all $\iota \in \Z_2^r \setminus \inner{\iota_1}$, it follows that every such codimension equals $\frac n 2$. It is not hard to see that there exists a decomposition $\Z_2^{\frac n 2} = V_1 \oplus \ldots \oplus V_{2^{r-1}}$ such that
	\begin{itemize}
	\item for all $\iota \in \Z_2^r$, the projection onto $V_i$ of image of $\iota$ has all $0$s or all $1$s, and
	\item The $V_i$ have equal dimensions.
	\end{itemize}
In particular, $\frac n 2$ is divisible by $2^{r-1}$, which proves that $2^r$ divides $n$.
\end{proof}

\smallskip\section{Dimension 10}\label{sec:dim10}\smallskip

The only two compact, simply connected, smooth $10$--manifolds known to admit positive sectional curvature are $\s^{10}$ and $\C\pp^5$. Equipped with the  standard metrics, each of these spaces admits $T^5$ symmetry. We partially recover a classification of these two spaces under the assumption of $T^3$ symmetry.

\begin{theorem}\label{thm:dim10}
Let $M^{10}$ be a closed, simply connected Riemannian manifold with positive sectional curvature and $T^3$ symmetry. One of the following occurs:
	\begin{itemize}
	\item $M$ is homeomorphic to $\s^{10}$.
	\item $\chi(M) = \chi(\C\pp^5)$, $H_i(M;\Z) = H_i(\C\pp^5;\Z)$ for $i \leq 3$, and $H^{10}(M;\Z)$ is generated by an element of the form $x^3 y$ with $x \in H^2(M;\Z)$ and $y \in H^4(M;\Z)$.
	\item $H_*(M;\Z) \cong H_*(\C\pp^5;\Z)$, and the cohomology is generated by some $z \in H^2(M;\Z)$ and $x \in H^4(M;\Z)$ subject to the relation $z^2 = m x$ for some $m \in \Z$.
	\end{itemize}
In any case, $H_2(M;\Z)$ is $0$ or $\Z$, $H_3(M) = 0$, and $\chi(M) = 2 + 4b_2(M) \in \{2,6\}$.
\end{theorem}

For context, we remark that, if $M$ is as in the theorem but has $T^4$ symmetry, work of Fang and Rong implies that $M$ is homeomorphic to $\s^{10}$ or $\C\pp^5$ (see \cite{FangRong05}). On the other hand, it follows from Rong--Su \cite{RongSu05} that a $T^1$ action on $M^{10}$ as in the theorem is sufficient to show $\chi(M) \geq 2$ (cf. Dessai \cite{Dessai11}). We also remark that this classification can be strengthened in the special case where $M$ is rationally elliptic or admits a biquotient structure (see Theorems \ref{thm:dim10elliptic} and \ref{thm:dim101416biquotient}).

We spend the rest of this section on the proof. Denote the torus by $T$. If some involution $\iota \in T$ has fixed-point set $M^\iota$ of codimension two, then $M$ is homeomorphic to $\s^{10}$ or $\C\pp^5$ by Lemma \ref{lem:cod2}. We assume therefore that $\cod(M^\iota) \geq 4$ for all non-trivial involutions $\iota\in T$. In this situation, there exist two involutions with fixed point sets of codimension four, so we are in the setting of the codimension four lemma, part 1. However, the arguments there do not suffice in dimension $10$. Fortunately, four is large enough relative to the dimension of $M$ in this case to force isotropy rigidity. The main consequence of this rigidity is the following:

\begin{lemma}[Containment lemma for dimension $10$]\label{lem:dim10containment}
Assume no involution $\iota \in T^3$ has fixed-point set of codimension two. There exist $x \in M^T$ and  independent involutions $\iota_1,\iota_2 \in T^3$ such that both $\cod\of{M^{\iota_i}_x} = 4$ and 
	$M^T \subseteq M^{\iota_1}_x \cup M^{\iota_2}_x \cup M^{\iota_1\iota_2}_x.$
\end{lemma}

\begin{proof}
Consider the map $\Z_2^3 \to \Z_2^5$ induced by the isotropy representation $T \to \SO(T_x M)$ at some fixed point $x \in M^T$. We claim that one of the following possibilities occurs.
	\begin{itemize}
	\item (Case 1) There exists $x \in M^T$ and independent involutions $\iota_1,\iota_2 \in T$ such that $\cod\of{M^{\iota_1}_x} = \cod\of{M^{\iota_2}_x} = 4$ and $M^{\iota_1}_x \cap M^{\iota_2}_x$ is not transverse.
	\item (Case 2) At every $x \in M^T$, there exist involutions $\iota_1,\iota_2 \in \Z_2^3$ such that
		\[\cod\of{M^\iota_x} = \left\{\begin{array}{rcl}
			4	&\mathrm{if}&	\iota\in\{\iota_1,\iota_2\}\\
			6	&\mathrm{if}&	\iota\not\in\langle\iota_1,\iota_2\rangle\\
			8	&\mathrm{if}&	\iota = \iota_1\iota_2\end{array}\right.\]
	\end{itemize}
Indeed, suppose Case 1 does not occur, and let $x \in M^T$. It is not possible for the image of every $\iota \in \Z_2^3$ to have weight at least three, so we may choose $\iota_1 \in \Z_2^3$ with $\cod\of{M^{\iota_1}_x} = 4$. Next, we may choose $\iota_2 \in \Z_2^3\setminus\inner{\iota_1}$ such that $M^{\iota_1}_x \cap M^{\iota_2}_x$ is transverse. Supposing for a moment that $\cod\of{M^{\iota_2}_x} = 6$, it follows that some $\iota \in \Z_2^3 \setminus\inner{\iota_1,\iota_2}$ exists such that $M^{\iota}_x$ has codimension four and intersects $M^{\iota_1}_x$ non-transversely. This is a contradiction since we assumed Case 1 does not occur. It follows that $\cod\of{M^{\iota_2}_x} = 4$. Finally, using again the fact that Case 1 does not occur, it follows that every $\iota \in \Z_2^3 \setminus\inner{\iota_1,\iota_2}$ has $\cod\of{M^\iota_x} = 6$. This concludes the proof of the claim.

If Case 1 occurs, then the lemma follows immediately from the containment lemma. We suppose now that Case 2 occurs and proceed by contradiction. If $z \in M^T$ does not lie in $M^{\iota_1}_x \cup M^{\iota_2}_x$, then $\cod\of{M^{\iota_1}_z} > 6$ and $\cod\of{M^{\iota_2}_z} > 6$ by Frankel's theorem. But this contradicts the fact that, at each point, there exists a unique involution whose fixed-point set component containing that point has codimension greater than six, so the proof is complete.
\end{proof}

We keep the notation $N_1 = M^{\iota_1}_x$ and $N_2 = M^{\iota_2}_x$ throughout the rest of the proof. The proof of Theorem \ref{thm:dim10} is finished below in the following two lemmas.

\begin{lemma}\label{lem:dim10cod4dk0}
If $N_1 \cap N_2$ is not transverse, or if $N_1$ or $N_2$ admits $T^3$ symmetry, then
	\begin{itemize}
	\item $M$ is homeomorphic to $\s^{10}$, or
	\item $\chi(M) = \chi(\C\pp^5)$, $H_i(M;\Z) = H_i(\C\pp^5;\Z)$ for $i \leq 3$, and $H^{10}(M;\Z)$ is generated by an element of the form $x^3 y$ with $x \in H^2(M;\Z)$ and $y \in H^4(M;\Z)$.
	\end{itemize}
\end{lemma}

\begin{proof}
We first calculate $H_2(M;\Z)$, $H_3(M;\Z)$, and $\chi(M)$.  Suppose first that $N_1$ and $N_2$ intersect transversely, and assume without loss of generality that $N_1$ is admits $T^3$ symmetry. It follows from Grove and Searle's diffeomorphism classification that $N_1$ is diffeomorphic to $\s^6$ or $\C\pp^3$. By either applying the same reasoning to $N_2$ or by using the connectedness lemma, it follows that $H_3(M;\Z) \cong H_3(N_2;\Z) = 0$ and $H_2(N_2;\Z) \cong H_2(M;\Z) \cong H_2(N_1;\Z)$, which is isomorphic to $0$ or $\Z$. Note that $N_1 \cap N_2 \to N_i \to M$ is $2$--connected by the connectedness lemma, so $M^{\iota_1\iota_2}_x = N_1 \cap N_2 = \s^2$. In particular, $M^T \subseteq N_1 \cup N_2$, so $\chi(M) = 2 + 4b_2(M) \in \{2,6\}$ by the inclusion-exclusion formula for Euler characteristics.

Suppose now that $N_1$ and $N_2$ do not intersect transversely. By the connectedness lemma, all two-fold intersections of $N_1$, $N_2$, and $N_{12} = M^{\iota_1\iota_2}_x$ coincide and equal $M^{\inner{\iota_1,\iota_2}}_x$. By the codimension two lemma, all four of these submanifolds have $2$--periodic integral cohomology, which means their third homology groups vanish, while their second homology groups coincide and equal $0$ or $\Z$. By the connectedness lemma, $H_3(M;\Z) = 0$ and $H_2(M;\Z)\in\{0,\Z\}$. Since $N_1 \cup N_2 \cup N_{12}$ contains $M^T$, the inclusion-exclusion formula for Euler characteristics implies $\chi(M) = 2 + 4b_2(M) \in \{2,6\}$.

This completes the calculation of $H_2(M;\Z)$, $H_3(M;\Z)$, and $\chi(M)$. We now complete the proof. Suppose first that $\chi(M) = 6$. Let $N$ be a six-dimensional submanifold of $M$ homotopy equivalent to $\C\pp^3$ such that the inclusion $N \to M$ is $3$--connected. By the naturality of cup products, it follows that the third power of a generator $x \in H^2(M;\Z)$ is non-zero and not a non-trivial multiple. By Poincar\'e duality, there exists $y \in H^4(M;\Z)$ such that $x^3 y$ generates $H^{10}(M;\Z)$.

Suppose now that $\chi(M) = 2$. By the calculations above, it follows that $M$ is $3$--connected. Moreover, we claim that the $T^3$--action on $M$ has the property that every $\Z_p \in T^3$ has fixed-point set equal to a mod $p$ homology sphere. Indeed, since $\Z_p \subseteq T^3$, we have $\chi\of{M^{\Z_p}} = \chi(M) = 2$, so it suffices to show that every component $P \subseteq M^{\Z_p}$ has vanishing odd Betti numbers. This clearly holds if $\dim(P) \leq 4$ by Synge's theorem, or if $\dim(P) = 8$ by Wilking's periodicity corollary. If $\dim(P) = 6$, it follows by Grove and Searle's diffeomorphism classification if $P$ is not fixed by a circle in $T^3$ and by the connectedness lemma if it is. The homeomorphism classification now follows from Wilking's spherical recognition theorem (Theorem \ref{thm:Wilking4.1}).
\end{proof}

\begin{lemma}\label{lem:dim10cod4dk1}
If $N_1$ and $N_2$ are fixed by circles in $T^3$ and intersect transversely, then
	\begin{itemize}
	\item $M$ is homeomorphic to $\s^{10}$.
	\item $H_*(M;\Z) \cong H_*(\C\pp^5;\Z)$, and the cohomology is generated by $H^2(M;\Z)$ and $H^4(M;\Z)$.
	\end{itemize}
\end{lemma}

\begin{proof}
Since $N_1$ and $N_2$ intersect transversely, the connectedness lemma implies that the inclusions $N_1 \cap N_2 \to N_i \to M$ are $2$--connected. In particular, $N_1 \cap N_2$ is diffeomorphic to $\s^2$ and equals the fixed point component $M^{\iota_1\iota_2}_x$. In particular, $M^T \subseteq N_1 \cup N_2$ by Lemma \ref{lem:dim10containment}. Note also that $H_2(M;\Z)$ is either finite or isomorphic to $\Z$ since the natural map $H_2(N_1 \cap N_2;\Z) \to H_2(M;\Z)$ is surjective.

First suppose $H_2(M;\Z)$ is finite. By the periodicity corollary, the Betti numbers of $M$ satisfy $b_2(M) = b_6(M) = 0$ and $b_4(M) = b_8(M) = 0$. By Poincar\'e duality, the Euler characteristic of $M$ is even and equals $2 - 2b_3(M) - b_5(M)$. But $\chi(M) > 0$, so $b_3(M) = b_5(M) = 0$. Hence $M$ is a rational sphere, and $N_1$ is as well by the connectedness lemma. Hence $M$ is homeomorphic to $\s^{10}$ by Lemma \ref{lem:cod4dk1Qsphere}.

Now suppose that $H_2(M;\Z) \cong \Z$. Since each $N_j$ is fixed by a circle in $T$, the connectedness lemma implies that $N_j \to M$ is $4$--connected. Applying the inclusion-exclusion formula for the Euler characteristic, we conclude that
	\[\chi(M) = \chi(N_1 \cup N_2) = \chi(N_1) + \chi(N_2) - \chi(N_1 \cap N_2) =  6-2b_3(M).\]
Comparing with the alternating sum of Betti numbers formula for $\chi(M)$, we conclude that $2 = 2b_4(M) - b_5(M)$. The periodicity corollary and Poincar\'e duality imply that $b_4(M) \leq 1$ and that $b_5(M)$ is even, so this equality implies that $b_4(M) = 1$ and $b_5(M) = 0$. Applying Poincar\'e duality and the periodicity corollary again, we conclude that $H_5(M;\Z) = 0$ and that $x \in H^4(M;\Z)$ and $z \in H^2(M;\Z)$ exist such that $xz$ generates $H^6(M;\Z)$, that $x^2 z$ generates $H^{10}(M;\Z)$, and hence that $x^2$ generates $H^8(M;\Z)$. Finally, it follows as in the proof of Lemma \ref{lem:4periodicZ2cohomology} that $H^3(M;\Z_2) = 0$, so $H^3(M;\Z) = 0$ and $M$ is as in the second conclusion of the lemma.
\end{proof}

\smallskip\section{Dimension 12}\label{sec:dim12}\smallskip

Let $C(6)$ denote the maximum Euler characteristic achieved by a closed, simply connected $6$--manifold that admits a Riemannian metric with positive sectional curvature and $T^2$ symmetry. Note that $C(6) < \infty$ by Gromov's Betti number estimate. In fact, we have the following:

\begin{lemma}\label{lem:C6}
The maximum Euler characteristic $C(6)$ of a closed, simply connected, positively curved manifold with $T^2$ symmetry satisfies $6 \leq C(6) \leq 14$.
\end{lemma}
\begin{proof}
The Wallach flag $M^6 = \SU(3)/T^2$ admits a metric with both positive curvature and an isometric $T^2$ action, so $C(6) \geq 6$. For the upper bound, let $M^6$ be a manifold as in the theorem. Suppose for a moment that some $g \in T^2$ has $\cod\of{M^g} = 2$. Let $N^4 \subseteq M^g$ denote a four-dimensional component. By the result of Hsiang and Kleiner, $N$ has $b_2(N) \leq 1$. By the connectedness lemma and Wilking's periodicity corollary, it follows that $M$ is homotopy equivalent to $\s^6$ or $\C\pp^3$.

Suppose therefore that every $g \in T^2$ has fixed-point set of codimension at least four. In this case, the fixed-point set of $T^2$ is made up of isolated fixed points. Moreover, this property implies that the isotropy representations of $T^2$ at fixed points are such that each fixed point projects to an extremal point in $M/T^2$. Since this space is a four-dimensional Alexandrov space, the number of extremal points is strictly less than $2^4$ (see Lebedeva \cite{Lebedeva15}). Since, on the other hand, the number equals $\chi(M)$, which is even, the proof is complete.
\end{proof}

The main result in dimension $12$ is presented in terms of the constant $C(6)$.

\begin{theorem}\label{thm:dim12}
Let $M^{12}$ be a closed, simply connected, positively curved manifold with $T^3$ symmetry. Either $\chi(M) \in \{2,4,\ldots,C(6)\}$ or $M$ is not spin and $7 \leq \chi(M) \leq \frac{7}{4} C(6)$. Moreover, the following hold.
	\begin{enumerate}
	\item if $C(6) = 6$, then $\chi(M) \in\{2,4,6,7,8,9\}$.
	\item if $M$ is rationally elliptic, then $\chi(M) \in\{2,4,6,7, 8, 9, 10, 12\}$.
	\end{enumerate}
In any case, the following hold for the signature and elliptic genus.
	\begin{enumerate}
	\item if $\chi(M) \leq 13$, then $|\sigma(M)| \in \{0,1\}$ according to the parity of $\chi(M)$.
	\item if $M$ is spin, then the elliptic genus is constant.
	\end{enumerate}
\end{theorem}

Recall that $2$, $4$, $6$, and $7$ are realized as Euler characteristics of positively curved manifolds with $T^3$ symmetry, namely, $\s^{12}$, $\HH\pp^3$, the Wallach flag $W^{12}$, and $\C\pp^6$. 
As for the possibilities of $\chi(M) > 7$, we note that there are many examples of non-negatively curved manifolds $M$ with isometric $T^3$--actions such that $\chi(M) \in \{8,9,10,12\}$ and $\sigma(M) \in \{0,1\}$. Indeed, one finds such examples among compact symmetric spaces of rank two (e.g., the Grassmannian $\SO(8)/\SO(2)\times\SO(6)$ or products of rank one spaces such as $\s^{12-2m} \times \C\pp^m$ or $\C\pp^2\times\HH\pp^2$) or among certain connected sums of rank one symmetric spaces (e.g., $\C\pp^6 \# \HH\pp^3$) endowed with Cheeger metrics (see \cite{Cheeger73}).

The proof of Theorem \ref{thm:dim12} takes the rest of the section. The bulk of the proof is contained in a sequence of lemmas that together prove the Euler characteristic calculation claimed in Theorem \ref{thm:dim12}. The signature calculation is then proved at the end of the section, and the elliptic genus calculation is proved in Section \ref{sec:EllipticGenus}.

We assume for the rest of the section that $M$ is a $12$--dimensional, compact, simply connected Riemannian manifold with positive curvature and an effective, isometric $T^3$ action.

\begin{lemma}\label{lem:dim12codim2and4part1}
If there exists an involution with fixed-point set of codimension two, or if there exist two involutions whose fixed-point sets have codimension four, then $M$ has the integral cohomology of $\s^{12}$, $\C\pp^6$, and $\HH\pp^3$.
\end{lemma}

This follows immediately from the codimension two and part 1 of the codimension four lemma (Lemmas \ref{lem:cod2} and \ref{lem:cod4part1}). The next case we consider is the following:

\begin{lemma}\label{lem:dim12codim4part2}
If the codimension of the fixed-point set is four for one involution but at least six for every other, then $M$ is homeomorphic to $\s^{12}$.
\end{lemma}

\begin{proof}
Note that $N_1 = M^{\iota_1}$ is connected by the codimension four lemma, part 2 (see Lemma \ref{lem:cod4part2}), so $H^i(M;\Z_2) = 0$ for $4 \leq i \leq 5$ by Theorem \ref{thm:Wilking5.1}. Moreover, if $\dk(N_1) = 1$, then this theorem implies $H^3(M;\Z_2) = 0$ and $H^6(M;\Z_2) = 0$ as well. Since $M^{\iota_1}$ is connected, we have $\chi(N_1) = \chi(M) = 2 + 2b_2(M)$. But $N_1$ admits an isometric $T^2$ action, so Dessai's Euler characteristic calculation in dimension $8$ implies that $\chi(N_1) \in \{2,3,5\}$. Hence $b_2(M) = 0$ and $M$ is a rational sphere. It now follows that $M$ is homeomorphic to $\s^{12}$ by Lemma \ref{lem:cod4dk1Qsphere}.

We may assume therefore that $\dk(N_1) = 0$. By Fang and Rong's homeomorphism classification, $N_1$ is homeomorphic to $\s^8$, $\C\pp^4$, or $\HH\pp^2$. By the connectedness lemma, $N_1 \to M$ is $5$--connected, so $N_1$ is homeomorphic to $\s^8$ and $M$ is $5$--connected. It follows that $H^6(M;\Z)$ is torsion-free with rank $b_6(M) = \chi(M) - 2 = \chi(N_1) - 2 = 0$, and hence that $M$ is again homeomorphic to $\s^{12}$.
\end{proof}

Before continuing with the proof, we remark that it might be surprising that the Euler characteristic of $\C\pp^6$ does not appear in Lemma \ref{lem:dim12codim4part2}. The example below shows that, while there exist $T^3$--actions on $\C\pp^6$ that realize the above isotropy data at one fixed point, they need not globally realize the isotropy data at all fixed points.

\begin{example}
Denote points in $\C\pp^6$ as equivalence classes $[z_0,z_1,\ldots,z_6]$ where $z_j\in\C$ such that $\sum |z_j|^2=1$. Define the actions of three circles on $\C\pp^6$ by the following three maps $\sone \to \PU(7)$:
	\begin{eqnarray*}
	w	&\mapsto&	\mathrm{diag}(1,w,w,1,1,1,1)\\
	w	&\mapsto&	\mathrm{diag}(1,w,1,w,w,1,1)\\
	w	&\mapsto&	\mathrm{diag}(1,1,1,w,1,w,w)
	\end{eqnarray*}
Note that, at the point $x = [1,0,0,0,0,0,0]$, the isotropy representation implies that the component at $x$ of the fixed-point set is four for one involution but is greater than four for all of the others. Also note, however, that the product of the involutions in the first and third circle factors of $T^3$ also has a codimension-four fixed-point set, so actually this action does not satisfy the hypotheses of Lemma \ref{lem:dim12codim4part2}.
\end{example}

To complete the proof of the Euler characteristic calculation claimed in Theorem \ref{thm:dim12}, we need to consider the case where every involution in the torus acts with fixed-point set of codimension at least six.

The key aspect of this case is the rigidity of the maps $\Z_2^3 \to \Z_2^6$ induced by the isotropy representations at fixed points of the torus action. More specifically, at each fixed point $x$, there exists a choice of basis for the tangent space at $x$ and a choice of $\rho,\sigma,\tau\in\Z_2^3$ such that the map $\Z_2^3 \to \Z_2^6$ induced by the isotropy representation at $x$ takes the form
	\begin{eqnarray*}
	\rho		&\mapsto	& (1,0,0, 1,1,0)\\
	\sigma	&\mapsto	& (0,1,0, 1,0,1)\\
	\tau		&\mapsto	& (0,0,1, 0,1,1)
	\end{eqnarray*}
Note that that
	\[\cod\of{M^\iota_x}	= \left\{\begin{array}{rcl}
	6	&\mathrm{if}&	\iota \in \{\rho, \sigma, \tau, \rho\sigma\tau\}\\
	8	&\mathrm{if}&	\iota \in \{\rho\sigma, \rho\tau, \sigma\tau\}
	\end{array}\right.\]
In particular, every $x \in M^T$ is an isolated fixed point, and we can associate to it a subgroup isomorphic to $\Z_2^2$ inside $\Z_2^3$, the complement of which has the property that every member $\iota$ has $\cod\of{M^\iota_x} = 6$. We call these complements ``clubs''. In particular, the club $C(x)$ at $x$ consists of the four members $\rho$, $\sigma$, $\tau$, and their three-fold product $\rho\sigma\tau$. In fact, it is an important property that the product of any three elements of a club is in that club. We call this the ``triple product property'' of clubs.

We analyze how these clubs overlap at distinct fixed points $x,y\in M^T$. One possibility is that the clubs at $x$ and $y$ coincide. By Frankel's theorem, this implies that $M^\iota_x = M^\iota_y$ for all $\iota \in C(x) = C(y)$. As it turns out, there is only one other possibility, namely, that the clubs $C(x)$ and $C(y)$ intersect in exactly two members. Indeed,
	\begin{itemize}
	\item $C(x) \cap C(y)$ contains at least one member, since each club consists of four of the seven non-trivial elements of $\Z_2^3$,
	\item if $C(x) \cap C(y)$ contains exactly one member, $\iota$, then the product of the three elements of $C(x) \setminus C(y)$ is both  equal to $\iota$ (by the triple product property) but not in $C(y)$ (since clubs are complements of subgroups), a contradiction, and
	\item if $C(x) \cap C(y)$ contains at least three members, then the three-fold product equals both the fourth member of $C(x)$ and that of $C(y)$ by the triple product property applied to both clubs, hence these clubs coincide in this case.
	\end{itemize}

Next, we analyze how clubs at three distinct fixed points $x,y,z\in M^T$ might over lap. There are two possibilities (up to relabeling the involutions in $\Z_2^3$). The first is
	\begin{eqnarray*}
	\mathrm{(Type~I)}	&~&\left\{\begin{array}{l}
	C(x) = \{\rho, \sigma\tau, \sigma, \rho\tau\}\\
	C(y) = \{\rho, \sigma\tau, \tau, \rho\sigma\}\\
	C(z) = \{\sigma, \rho\tau, \tau, \rho\sigma\}
	\end{array} \right.
	\end{eqnarray*}
The second is (Type II), in which $C(x)$ and $C(y)$ are exactly the same as above, and $C(z)$ contains $\rho\sigma\tau$ as well as exactly one involution from each of the sets $C(x) \cap C(y)$, $C(x) \setminus C(y)$, and $C(y) \setminus C(x)$. Note however these choices are neither unique nor arbitrary since $C(z)$ satisfies the triple product property. We omit the proof, as it follows simply from further analysis using the triple product property of clubs.

Using this club analysis, we claim the following.

\begin{lemma}[Club analysis]\label{lem:ClubAnalysis}
One of the following occurs:
	\begin{enumerate}
	\item There exists $\iota \in \Z_2^3$ whose fixed-point set is connected, has dimension six, and admits an effective, isometric $T^2$--action.
	\item There exist three clubs with Type I intersection data, and $M$ is not spin.
	\end{enumerate}
\end{lemma}

\begin{proof}
It might happen that some involution $\iota$ is in every club. In this case, Frankel's theorem implies that $M^\iota$ has dimension six and is connected. Moreover, by the rigidity of the isotropy representation, $M^\iota$ is fixed by at most a circle, so in this case the first possibility of the conclusion occurs.

We claim such an involution $\iota$ exists in each of the following three cases:
	\begin{enumerate}
	\item There exists exactly one club.
	\item There exist exactly two clubs.
	\item There exist at least three clubs, and every subset of three has intersection data of Type II.
	\end{enumerate}
Indeed, in each of the first two cases, the club analysis above implies that some involution is in all clubs. In the last case, we verify this as follows. Suppose $C(x) = \{\iota_1,\iota_2,\iota_3,\iota_4\}$ and $C(y) = \{\iota_1,\iota_2,\iota_5,\iota_6\}$ are two of the clubs. We assume that no involution is in every club and proceed by contradiction. Since the intersection data of every club with $C(x)$ and $C(y)$ is of Type II, there exist clubs $C(z)$ and $C(w)$ such that one contains $\iota_1$ but not $\iota_2$ and vice versa for the other club. Next our assumption implies that the clubs $C(y)$, $C(z)$, and $C(w)$ have intersection data of Type II. A general property of such triples of clubs is that their union contains all seven of the non-trivial involutions in $\Z_2^3$. In particular, $C(z) \cup C(w)$ contains both $\iota_3$ and $\iota_4$. By applying the same line of reasoning to the triple of clubs $C(x)$, $C(z)$, and $C(w)$, we conclude that $C(z) \cup C(w)$ contains $\iota_5$ and $\iota_6$. Altogether we have that $C(z) \cup C(w)$ contains all seven non-trivial involutions in $\Z_2^3$. This contradicts the fact that any two clubs either coincide or intersect in exactly two elements. This completes the proof of the claim.

Assuming the claim now, there exist at least three (distinct) clubs, and among these there exist three with intersection data of Type I. We see immediately in this case that $M$ is not spin. Indeed, fixed-point sets of involutions on spin manifolds have all components of codimension congruent to $c$ modulo four, for some $c \in \{0,2\}$. In our setting, no involution is in every club, so any member $\iota \in C(x)$ has fixed point components of codimension $6$ and $8$.
\end{proof}

We require one more lemma whose proof relies on further club analysis, together with the {equivariant} diffeomorphism classification of Grove and Searle.

\begin{lemma}\label{lem:GroveSearleTrick}
Assume $M$ is as in Theorem \ref{thm:dim12}, and that every involution in the torus $T$ has fixed-point set of codimension at least six. Suppose $\iota \in T$ is an involution and $P \subseteq M^\iota$ is a component of dimension four.
	\begin{itemize}
	\item If $P$ is $\s^4$, then the two fixed-points in $P^T = P \cap M^T$ have the same club.
	\item If $P$ is $\C\pp^2$, then the three fixed points in $P^T$ represent three distinct clubs.
	\end{itemize}
\end{lemma}

\begin{proof}
Fix $\iota \in \Z_2^3$ and a four-dimensional component $P \subseteq M^\iota$. By the rigidity of the isotropy maps, the induced $T^3$--action on $P$ has one-dimensional kernel that does not contain any other involution. Let $\overline T$ denote the two-dimensional torus equal to the quotient of $T^3$ by the kernel of this induced action. The induced action of $\overline T$ on $P$ is equivariant to a linear action on $\s^4$ or $\C\pp^2$.

Applying isotropy rigidity again, it follows that, for all $p \in P^{\overline{T}} = P \cap M^T$, there exist $\iota_1$ and $\iota_2$ such that the isotropy map at $p$ takes the form
	\begin{eqnarray*}
	\iota		&\mapsto&	(1, 1, 1, 1, 0, 0)\\
	\iota_1	&\mapsto&	(*, *, *, *, 1,0)\\
	\iota_2	&\mapsto&	(*, *, *, *, 0, 1)
	\end{eqnarray*}
By the rigidity of the isotropy at $p$, it follows that the club at $p$ is given by
	\[C(p) = \{\iota_1, \iota\iota_1, \iota_2, \iota\iota_2\}.\]

To complete the proof, first suppose $P$ is a sphere, and denote the fixed points in $P \cap M^T$ by $p_1$ and $p_2$. Since the $\overline T$--action on $P$ is equivariant to a linear action on $\s^4$, it follows that $P^{\tau}_{p_1}$ and $P^{\tau}_{p_2}$ have the same dimension for all involutions $\tau$ in $\overline T$. In particular, the $\iota_1$ and $\iota_2$ from the previous paragraph are the same for $p_1$ and $p_2$, and hence the clubs coincide at these points.

Now suppose $P$ is $\C\pp^2$. We use the fact that the $\overline T$--action on $P$ is equivariant to a linear action. In particular, each involution $\iota \in \overline T$ has the property that $P^\iota$ consists of a copy of $\s^2$ together with an isolated point. In particular, if $p_1$ and $p_2$ are two points in $P^{\overline T}$ with the same club, and if $\iota_1$ and $\iota_2$ are the involutions as above such that
	$C(p_1) = C(p_2) = \{\iota_1, \iota\iota_1, \iota_2, \iota\iota_2\},$
then the third point, $p_3 \in P^{\overline T}$, is an isolated fixed point of the actions of $\iota_1$ and $\iota_2$ on $P$. But then the product $\iota_1\iota_2$ acts on $T_{p_3} P$ as the identity and hence fixes $P$. This contradicts the rigidity of the isotropy maps.
\end{proof}

We now complete the proof of the Euler characteristic calculation claimed in Theorem \ref{thm:dim12}. We do this is three steps (see Lemmas \ref{lem:dim12codim6}, \ref{lem:dim12codim6Cis6}, and \ref{lem:dim12codim6elliptic}).

\begin{lemma}\label{lem:dim12codim6}
If every involution has fixed-point set of codimension at least six, then one of the following occurs:
	\begin{enumerate}
	\item $\chi(M) \in \{2, 4,6,\ldots, C(6)\}$, or
	\item $7 \leq \chi(M) \leq \frac{7}{4} C(6)$ and $M$ is not spin.
	\end{enumerate}
\end{lemma}

\begin{proof}
If there exists $\iota \in \Z_2^3$ with connected fixed-point set of dimension six, then we have $\chi(M) = \chi(M^\iota) \in \{2,4,6,\ldots,C(6)\}$ by Lemma \ref{lem:C6}. Suppose then, as in Lemma \ref{lem:ClubAnalysis}, that distinct clubs, $C(x)$, $C(y)$, and $C(z)$, exist and have Type I intersection data. We recall the notation from the proof of that lemma. Furthermore, we denote the non-trivial involutions by $\iota_0 = \rho\sigma\tau$, $\iota_1 = \rho$, $\iota_3 = \sigma$, $\iota_5 = \tau$, and $\iota_{2i} = \iota_0\iota_{2i-1}$ for $2i \in \{2,4,6\}$. In particular,
	\[C(x) \cap C(y) = \{\iota_1,\iota_2\}, ~~~
	C(x) \cap C(z) = \{\iota_3,\iota_4\}, ~~~
	C(y)\cap C(z) = \{\iota_5,\iota_6\},\]
and $\iota_0$ is not in any of these clubs. Let $N_1,\ldots,N_6$ denote the (unique, by Frankel) components of dimension six of the fixed-point sets of $\iota_1,\ldots,\iota_6$, respectively. If $\iota_0$ is in some club, let $N_0$ denote the six-dimensional component of its fixed-point set; otherwise, let $N_0$ denote the empty set.

Let $X\subseteq M^T$ denote the set of fixed points whose club equals $C(x)$. Define $Y$ and $Z$ similarly, and let $W$ denote the fixed points of $T$ whose club is distinct from $C(x)$, $C(y)$, and $C(z)$. Note the following facts:
	\begin{itemize}
	\item $X$, $Y$, $Z$, and $W$ partition the fixed-point set $M^T$.
	\item $N_{2i-1} \cap W$ and $N_{2i} \cap W$ partition $W$ for all $i \in \{1,2,3\}$. Indeed, if $w \in N_{2i}$, then $\iota_{2i} \in C(w)$ and hence $\iota_{2i-1} = \iota_0 \iota_{2i} \not\in C(w)$. Moreover, if $w \not\in N_{2i} \cup N_{2i-1}$, then $\iota_{2i}$ and $\iota_{2i-1}$ are not in $C(w)$, and hence $\iota_0 = \iota_{2i-1}\iota_{2i} \not\in C(w)$, a contradiction.
	\item Each $w \in W$ lies in exactly one or all three of $N_1$, $N_3$ or $N_5$. Indeed, this follows from the previous fact together with the triple product property and the fact that $\iota_2\iota_4\iota_6$ is the identity.
	\end{itemize}

We denote $\chi(X)$ by $|X|$ and similarly for the orders of $Y$, $Z$, and $W$. Observe that $w \in W$ if and only if $\iota_0 \in C(w)$, hence $W = (N_0)^T$ and
	\begin{eqnarray}\label{eqn:L1}
	|W| = \chi(N_0).
	\end{eqnarray}
Next, note that $X \cup Y = (N_1 \cap N_2)^T$ and that, similarly, $X \cup Z = (N_3\cap N_4)^T$ and $Y\cup Z = (N_5 \cap N_6)^T$. Hence
	\begin{eqnarray}
	|X|+|Y| &=& \chi(N_1 \cap N_2),\\\label{eqn:L2}
	|X|+|Z| &=& \chi(N_3 \cap N_4),\\\label{eqn:L3}
	|Y|+|Z| &=& \chi(N_5 \cap N_6).\label{eqn:L4}
	\end{eqnarray}
Next, note that $X \cup Y \cup W = (N_1 \cup N_2)^T$ and that we have similar statements for $N_3 \cup N_4$ and $N_5\cup N_6$. Hence
	\begin{eqnarray}
	|X|+|Y|+|W| 	&=&	\chi(N_1) + \chi(N_2) - \chi(N_1 \cap N_2),\\\label{eqn:L5}
	|X|+|Z|+|W|	&=& \chi(N_3) + \chi(N_4) - \chi(N_3 \cap N_4),\\\label{eqn:L6}
	|Y|+|Z|+|W|	&=& \chi(N_5) + \chi(N_6) - \chi(N_5 \cap N_6).\label{eqn:L7}
	\end{eqnarray}
Adding together Equations (\ref{eqn:L1})--(\ref{eqn:L7}), we conclude
	\begin{eqnarray}
	4 \chi(M) = \sum_{j=0}^6 \chi(N_i).\label{eqn:4chi}
	\end{eqnarray}

We first use this equation to bound $\chi(M)$ from below. Since no $\iota_j$ is in every club, every $M^{\iota_j}$ is composed of the six-manifold $N_j$ together with a disjoint union of $4$--manifolds, each of which has Euler characteristic at least two. Hence $\chi(M) \geq \chi(N_j) + 2$ for all $1 \leq j \leq 6$. In addition, $M^{\iota_0}$ contains $N_0$ as well as $X$, $Y$, and $Z$, so $\chi(M) \geq \chi(N_0) + 3$. Hence Equation (\ref{eqn:4chi}) implies
	\[4\chi(M) 	\leq (\chi(M) - 3) + \sum_{j=1}^6 \of{\chi(M) - 2},\]
which implies that $\chi(M) \geq 5$. In fact, $\chi(M) = 5$ would imply that $\chi(N_j) \leq 3$ for all $1 \leq j \leq 6$. But $\chi(N_j)$ is both positive and even, so $\chi(N_j) = 2$ for all $1 \leq j \leq 6$. Since $\chi(N_0) \leq \chi(M) - 3 = 2$, Equation \ref{eqn:4chi} implies $4\chi(M) \leq 7(2) = 14$, a contradiction. In fact, $\chi(M) = 6$ also cannot occur. Indeed, suppose $\chi(M) = 6$. Then $\chi(N_0) \leq 3$, so $\chi(N_0) = 2$. But now $\chi(M^{\iota_0} \setminus N_0) = 4$, a contradiction to Lemma \ref{lem:GroveSearleTrick}, which implies that these four points must come in two pairs such that, in each pair, the two clubs are the same. But these four points together represent three clubs (their union is $X \cup Y \cup Z)$, a contradiction.

To complete the proof of the lemma, it suffices to prove $\chi(M) \leq \frac 7 4 C(6)$. By the rigidity of the isotropy maps, $N_i$ admits an effective, isometric $T^2$--action and hence has $\chi(N_i) \leq C(6)$ for all $1 \leq i \leq 6$. The same estimate on $\chi(N_0)$ holds, even if it is empty. The upper bound now follows from Equation \ref{eqn:4chi}.
\end{proof}

\begin{lemma}\label{lem:dim12codim6Cis6}
Let $M$ be as in Lemma \ref{lem:dim12codim6}. If $C(6) = 6$, then $\chi(M) \leq 9$.
\end{lemma}

\begin{proof}
We keep the notation from above, but now we assume $\chi(N_i) \leq 6$ for all $0 \leq i \leq 6$. Equation \ref{eqn:4chi} implies $4 \chi(M) \leq 42$, and hence $\chi(M) \leq 10$.

Suppose that $\chi(M) = 10$. Equation (\ref{eqn:4chi}) implies $\chi(N_0) \geq 4$. Moreover, if $\chi(N_0) = 6$, then we have $|X| = |Y| = 1$ and $|Z| = 2$ without loss of generality, and one can show that $x$ and $y$ can be replaced by suitable choices of $w_1$ and $w_2$ in $W$ so that the three clubs $C(z)$, $C(w_1)$, and $C(w_2)$ are distinct, have intersection data of Type I, and have the property that the involution not in $C(z) \cup C(w_1) \cup C(w_2)$ has maximal component of Euler characteristic four. In other words, we may assume without loss of generality that $\chi(N_0) = 4$. By Equation (\ref{eqn:4chi}), $\chi(N_j) = 6$ for all $1 \leq j \leq 6$. By Equations (\ref{eqn:L2})--(\ref{eqn:L7}), it follows that $|X| = |Y| = |Z| = 2$. Write $W = \{w_1,w_2,w_3,w_4\}$. Up to relabeling the $w_i$, we have that
	\begin{eqnarray*}
	N_1^T &= X \cup Y \cup \{w_1,w_2\}.
	\end{eqnarray*}
Now every $w \in W$ lies in exactly one or three of $N_1$, $N_3$, and $N_5$. Without loss of generality, this implies that
	\begin{eqnarray*}
	N_3^T &= X \cup Z \cup \{w_1,w_3\},\\
	N_5^T &= Y \cup Z \cup \{w_1,w_4\}.
	\end{eqnarray*}
In particular, $\iota_3$ is in the club $C(w_3)$ but not $C(w_4)$, so $C(w_3) \neq C(w_4)$.

On the other hand, consider the fixed-point set $M^{\iota_1}$. One component is $N_1$, which has Euler characteristic six. The others are closed, oriented, positively curved $4$--manifolds. By Lemma \ref{lem:GroveSearleTrick}, we have that $M^{\iota_1} = N_1 \cup P \cup Q$ where $P$ and $Q$ are diffeomorphic to $\s^4$, and where the two fixed-points in $P$ have equal clubs and likewise for the fixed points of $Q$. In particular, the four points in $(P\cup Q)^T$ represent only two clubs. But $(P \cup Q)^T$ contains the two points in $Z$, as well as $w_3$ and $w_4$. These points represent three distinct clubs, so we have the desired contradiction.
\end{proof}

To complete the proof of the Euler characteristic calculation claimed in Theorem \ref{thm:dim12}, it suffices to prove the following.

\begin{lemma}\label{lem:dim12codim6elliptic}
Let $M$ be as in Lemma \ref{lem:dim12codim6}. If $M$ is rationally elliptic, then $\chi(M) \leq 10$ or $\chi(M) = 12$.
\end{lemma}

\begin{proof}
We keep the notation from the proof of Lemma \ref{lem:dim12codim6}. Consider one of the submanifolds $N_i$ in Equation \ref{eqn:4chi}. If $N_i$ admits an effective, isometric $T^3$ action, then $\chi(N_i) \in \{2,4\}$ by Grove and Searle's classification. Otherwise, $N_i$ is a fixed point component of a circle action on $M$. Since $M$ is rationally elliptic, it follows that $N_i$ is rationally elliptic. Since $N_i$ is a simply connected, closed manifold, it follows that $\chi(N_i) \leq \chi(\s^2 \times \s^2 \times \s^2) = 8$. Since this estimate holds for all $0 \leq i \leq 6$, Equation \ref{eqn:4chi} implies that $\chi(M) \leq 14$.

To complete the proof, it suffices to observe that no rationally elliptic, simply connected, closed manifold of dimension $12$ can have Euler characteristic $11$, $13$, or $14$. For a proof, one can use the existence of a pure minimal model for such a space and apply the properties listed in the beginning of the proof of Theorem \ref{thm:PossibleTuples}. The calculation also immediately follows from the results in Table \ref{tab12}.
\end{proof}

This completes the proof of the Euler characteristic calculation. For the elliptic genus, see Section \ref{sec:EllipticGenus}. It suffices to compute the signature.

\begin{proof}[Proof of Theorem \ref{thm:dim12}, signature calculation]
First, if $M$ is as in Lemmas \ref{lem:dim12codim2and4part1} or \ref{lem:dim12codim4part2}, then $M$ in an integral cohomology $\s^{12}$, $\C\pp^6$, or $\HH\pp^3$. In each of these cases, it follows that $|\sigma(M)|$ is $0$ or $1$, according to the parity of $\chi(M)$.

Second, if $M$ is as in the first possibility of Lemma \ref{lem:ClubAnalysis}, then there exists and involution $\iota$ such that $M^\iota$ is a closed, connected, simply connected $6$--manifold. Hence $\sigma(M) = \sigma(M^\iota) = 0$. Note that this is consistent with the fact that $\chi(M)$ is even in this case by the proof of Lemma \ref{lem:dim12codim6}.

We may therefore assume $M$ is as in the second possibility of Lemma \ref{lem:ClubAnalysis}. In particular, $M$ is not spin, no fixed-point set $M^{\iota_i}$ is connected, and $7 \leq \chi(M) \leq \frac{7}{4} C(6)$. Since $C(6) \leq 14$, it follows that $7 \leq \chi(M) \leq 24$.

We cannot calculate the signature in all cases, so we further assume $\chi(M) \leq 13$. The proof of Lemma \ref{lem:dim12codim6} (in particular, Equation \ref{eqn:4chi}) implies $\chi(M) - \chi(N_i) \in \{2, 3, 4, 5\}$ for some $0 \leq i \leq 6$, where the $N_i$ are the $6$--dimensional fixed-point components of the $\iota_i$ as in the proof of Lemma \ref{lem:dim12codim6}. In particular, we have the following possibilities:
	\begin{itemize}
	\item $M^{\iota_i} = N_i \cup \s^4$, and so $\sigma(M) = \sigma(\s^4) = 0$,
	\item $M^{\iota_i} = N_i \cup \C\pp^2$, and so $|\sigma(M)| = |\sigma(\C\pp^2)| = 1$,
	\item $M^{\iota_i} = N_i \cup \s^4 \cup \s^4$, and so $\sigma(M) = 0$, or
	\item $M^{\iota_i} = N_i \cup \s^4 \cup \C\pp^2$, and so $|\sigma(M)| = 1$.
	\end{itemize}
Hence the signature is $0$ or $\pm 1$ according to the parity of $\chi(M)$.
\end{proof}

We remark that an extension of this argument shows that $|\sigma(M)| = 1$ if there exists some $N_i$ with $\chi(M) - \chi(N_i) = 7$. Using this together with the fact that $|\sigma(M)| \equiv \chi(M) \bmod{2}$, one can further compute that $|\sigma(M)| = 1$ if $\chi(M) \in \{15,17,19\}$.

\smallskip\section{Dimension 14}\label{sec:dim14}\smallskip

The only simply connected, smooth, closed manifolds of dimension $14$ known to admit positive sectional curvature are $\s^{14}$ and $\C\pp^7$. The following result provides a sharp calculation of the Euler characteristic and the second and third homology groups of a positively curved $14$--manifold in the presence of $T^4$ symmetry.

\begin{theorem}\label{thm:dim14}
If $M^{14}$ is a closed, simply connected, positively curved Riemannian manifold with $T^4$ symmetry, then one of the following occurs.
\begin{itemize}
	\item $M$ is $3$--connected and $\chi(M) = 2$.
	\item $H_*(M;\Z) \cong H_*(\C\pp^7;\Z)$, and $H^*(M;\Z)$ is generated by some $z \in H^2(M;\Z)$ and $x \in H^4(M;\Z)$ subject to the relation $z^2 = m x$ for some $m \in \Z$.
	\end{itemize}
\end{theorem}

Note that, if $T^4$ is replaced by $T^5$ in this statement, then $M$ is tangentially homotopy equivalent to $\s^{14}$ or $\C\pp^7$ (see \cite{Wilking03,DessaiWilking04}). We proceed to the proof.

\begin{lemma}\label{lem:dim14codim4}
If a non-trivial involution $\iota_1 \in T^4$ exists such that $\cod\of{M^{\iota_1}} \leq 4$, then one of the following occurs:
	\begin{itemize}
	\item $M$ is homeomorphic to $\s^{14}$.
	\item $H_*(M;\Z) \cong H_*(\C\pp^7;\Z)$, and $H^*(M;\Z)$ is generated by some $z \in H^2(M;\Z)$ and $x \in H^4(M;\Z)$ subject to the relation $z^2 = m x$ for some $m \in \Z$.
	\end{itemize}
\end{lemma}

As the proof indicates, a weak version of this lemma is an easy consequence of the codimension four lemma. For the stronger conclusion, we require our calculation in dimension $10$ (see Theorem \ref{thm:dim10}).

\begin{proof}
Let $M^{\iota_1}_x$ be a fixed point component of an involution of maximum dimension. By the codimension two lemma, we may assume $\cod(M^{\iota_1}_x) = 4$. By maximality, $M^{\iota_1}_x$ has $T^3$ (or $T^4$) symmetry. Theorem \ref{thm:dim10} implies that $b_2(M^{\iota_1}_x)=1$ or $M^{\iota_1}_x$ is homeomorphic to $\s^{10}$. In the latter case, the connectedness lemma implies $M$ is homeomorphic to $\s^{14}$. In the former case, if $M^{\iota_1}_x$ has $T^4$ symmetry, then it is a cohomology $\C\pp^5$ by Wilking's homotopy classification, and $M$ is a cohomology $\C\pp^7$ by the connectedness lemma. We assume now that $H_2(M;\Z) \cong \Z$ and that $M^{\iota_1}_x$ is fixed by a circle in $T^4$.

Consider the map $\Z_2^4 \to \Z_2^7$ induced by the isotropy representation of $T^4$ at $x$. There exists $\iota_2 \in \Z_2^4 \setminus \inner{\iota_1}$ such that $\cod\of{M^{\iota_2}_x} \leq 6$, and the codimension four lemma, part 1, applies. The lemma follows if $\cod(M^{\iota_2}_x) \leq 4$, so we may assume $\cod(M^{\iota_2}_x) = 6$. Returning to the isotropy map $\Z_2^4 \to \Z_2^7$, we see that we may choose $\iota_2$ so that, in addition, $M^{\iota_1}_x$ and $M^{\iota_2}_x$ intersect non-transversely. From the previous paragraph, we have without loss of generality that $M^{\iota_2}_x$ has $T^4$ symmetry and hence is diffeomorphic to $\C\pp^4$ by the diffeomorphism classification of Grove and Searle. Since the intersection of $M^{\iota_1}_x$ and $M^{\iota_2}_x$ is not transverse, the intersection of any two of $M^{\iota_1}_x$, $M^{\iota_2}_x$, and $M^{\iota_1\iota_2}_x$ equals the $6$--dimensional manifold $M^{\inner{\iota_1,\iota_2}}_x$. By the connectedness lemma, $M^{\iota_2}_x$, $M^{\iota_1\iota_2}_x$, and $M^{\inner{\iota_1,\iota_2}}_x$ are cohomology complex projective spaces. By the containment lemma,
	\[M^{T^4} \subseteq M^{\iota_1}_x \cup M^{\iota_2}_x \cup M^{\iota_1\iota_2}_x,\]
so the inclusion-exclusion property of Euler characteristics implies
	\[\chi(M) - \chi(M^{\iota_1}_x) = 5 + 5 - 3(4) + (4) = 2.\]
On the other hand, the periodicity corollary together with the fact that $M^{\iota_1}_x$ is fixed by a circle in $T^4$ imply that
	\[\chi(M) - \chi(M^{\iota_1}_x) = -b_5(M) + 2b_6(M) - b_7(M) \leq 2b_2(M) = 2.\]
Combining the previous two computations, we conclude that $M$ has the Betti numbers of $\C\pp^7$. Applying the periodicity corollary again, we conclude that $\Z \cong H^2(M;\Z) \cong H^6(M;\Z)$ and hence that $H_5(M;\Z) = 0$. Recalling that $H_3(M;\Z) = 0$, it follows that $H_*(M;\Z) \cong H_*(\C\pp^7;\Z)$. Finally, the connectedness lemma applied to the inclusion $M^{\iota_2}_x \to M$ and Poincar\'e duality imply that $H^*(M;\Z) \cong H^*(\C\pp^7;\Z)$.
\end{proof}

To complete the proof of the theorem, it suffices to prove the following:

\begin{lemma}
If the fixed-point set of every involution has codimension at least $6$, then $M$ is $3$-connected and has Euler characteristic two.\end{lemma}

The proof in this case is similar to, but not as hard as, the corresponding lemma in dimension $12$. As there, the key aspect is the rigidity of the maps $\Z_2^4 \to \Z_2^7$ induced by the isotropy representations at fixed points $x \in M^{T^4}$.

\begin{proof}
Let $T$ denote a torus of rank four acting effectively and isometrically on $M$.
Since every non-trivial involution $\iota \in T$ has $\cod\of{M^\iota} \geq 6$, we have rigidity in the isotropy representation $\Z_2^4 \to \Z_2^7 \subseteq \SO(T_x M)$. Indeed, for every $x\in M^T$, there exists a basis of $T_x M$ and a generating set of involutions $\iota_1,\ldots,\iota_4\in\Z_2^4$ such that the map $\Z_2^4 \to \Z_2^7$ can be represented as follows:
	\begin{eqnarray*}
	\iota_1	&\mapsto&	(1,1,0,1,0,0,0),\\
	\iota_2	&\mapsto&	(1,0,1,0,1,0,0),\\
	\iota_3	&\mapsto&	(0,1,1,0,0,1,0),\\
	\iota_4	&\mapsto&	(1,1,1,0,0,0,1).
	\end{eqnarray*}

In particular, at every $x \in M^T$, there exist seven distinct involutions whose fixed-point component containing $x$ has codimension six, seven with codimension eight, and exactly one with codimension $14$.

As a consequence of this rigidity and Frankel's theorem, if $\iota_1$ and $\iota_2$ are distinct involutions such that $N_1 = M^{\iota_1}_x$ and $N_2 = M^{\iota_2}_x$ have codimension six, then $N_1 \cup N_2$ contains $M^T$. By the connectedness lemma, $N_1 \cap N_2$ is one--connected and $b_2(N_1 \cap N_2) \geq b_2(N_i) = b_2(M)$ for $i \in \{1,2\}$. By the rigidity of the isotropy representation, $N_1$ and $N_2$ have $T^3$ symmetry and $N_1 \cap N_2$ has $T^2$ symmetry. By the classifications of Grove--Searle and Fang--Rong, all three submanifolds are homotopy spheres or complex or quaternionic projective spaces. By the inclusion-exclusion property of Euler characteristics, we obtain the estimate
	\[\chi(M) = \chi(N_1) + \chi(N_2) - \chi(N_1 \cap N_2) \leq 7.\]
In fact, $\chi(M) \equiv b_7(M) \equiv 0 \bmod{2}$ by Poincar\'e duality, so $\chi(M) \in \{2, 4, 6\}$.

In particular, the number of fixed points of $T$ is less than the number of involutions $\iota \in T$ with $\cod\of{M^\iota_x} = 6$. At most one of these involutions can have $\cod\of{M^\iota_y} = 14$ at any one fixed point $y \in M^T$, so there exists some involution $\iota$ with $\cod\of{M^\iota_y} < 14$ for all $y \in M^T$. By the rigidity of the isotropy representation and Frankel's theorem, $M^\iota$ is connected. Applying Fang and Rong's classification again, we conclude that $M^\iota$ either is homeomorphic to $\s^8$ or has odd Euler characteristic. Since $\chi(M) = \chi(M^\iota)$, we conclude that $M^\iota$ is homeomorphic to $\s^8$,  $\chi(M) = 2$, and  $M$ is $3$--connected.

\end{proof}

\smallskip\section{Dimension 16}\label{sec:dim16}\smallskip

The four known, simply connected, compact, positively curved examples in dimension $16$ are the rank one symmetric spaces, $\s^{16}$, $\C\pp^2$, $\HH\pp^4$, and $\C\pp^8$, and each of these admits a positively curved metric with $T^4$ symmetry. 

\begin{theorem}\label{thm:dim16}
If $M^{16}$ is a closed, simply connected Riemannian manifold with positive sectional curvature and $T^4$ symmetry, then one of the following occurs:
	\begin{itemize}
	\item $\chi(M) = 2$ and $\sigma(M) = 0$.
	\item $\chi(M) = 3$, $\sigma(M) = \pm 1$, and $M$ is $2$--connected.
	\item $\chi(M) = 5$, $\sigma(M) = \pm 1$, $H_{2+4i}(M;\Z) = 0$ for all $i$, and $b_4(M) = 1 + 2b_3(M)$.
	\item $\chi(M) = 9$, $\sigma(M) = \pm 1$, $H_2(M;\Z) = \Z$, $H_3(M;\Z) = 0$, $H_4(M;\Z) = \Z$, and $H^{16}(M;\Z)$ is generated by an element of the form $x^5 y$ with $x \in H^2(M;\Z)$ and $y \in H^6(M;\Z)$. Moreover, $M$ is not spin.
	\end{itemize}
If moreover $M$ is spin, then the elliptic genus is constant.
\end{theorem}

The proof of the last claim is in Section \ref{sec:EllipticGenus}. The proof of the first claim takes the rest of the section and is contained in Lemmas \ref{lem:dim16codim4}, \ref{lem:dim16codim6}, and \ref{lem:dim16codim8}.

Let $M$ be as Theorem \ref{thm:dim16}, let $T$ denote a torus of rank four acting effectively and isometrically on $M$, and let $\Z_2^4 \subset T$ denote the subgroup of involutions.

\begin{lemma}\label{lem:dim16codim4}
If there exists $\iota_1 \in \Z_2^4$ with $\cod\of{M^{\iota_1}} \in \{2, 4\}$, then
	\begin{enumerate}
	\item $M$ is homeomorphic to $\s^{16}$,
	\item $M$ is homotopy equivalent to $\C\pp^8$, or
	\item $\chi(M) = 5$, $\sigma(M) = \pm 1$, $H_{2+4i}(M;\Z) = 0$ for all $i$, and $b_4(M) = 1 + 2b_3(M)$.
	\end{enumerate}
\end{lemma}

The proof is a bit involved, and we would like to illustrate by example some of the structure we recover in the most difficult case of the proof.

\begin{example}
Denote points in $\HH\pp^4$ as equivalence classes $[q_0,q_1,\ldots,q_4]$ where $q_s \in\HH$ such that $\sum |q_s|^2=1$. Define the actions of four circles on $\HH\pp^4$ by the following three maps $\sone \to \Sp(5)$:
	\begin{eqnarray*}
	w	&\mapsto&	\mathrm{diag}(w,1,1,1,1)\\
	w	&\mapsto&	\mathrm{diag}(w,w,w,w,w)\\
	w	&\mapsto&	\mathrm{diag}(1,1,w,1,w)\\
	w	&\mapsto&	\mathrm{diag}(1,1,1,w,w)
	\end{eqnarray*}
The involution $\iota_1$ in the first circle fixes a component $\HH\pp^3$ of codimension four. The involution $\iota_2$ in the second circle acts trivially, so really we consider the action by the quotient of the second circle by $\{\pm 1\}$. The involution in this circle is then represented by $(i,i,i,i,i)$, and its fixed point set is a $\C\pp^4$. Notice that all fixed points of the $T^4$ action are isolated, and that they come in two types. If $\iota_1,\ldots,\iota_4$ denote the involutions in the four circles above, then at the point $z = [1,0,0,0,0]$ one has
	\[\cod(M^\iota_z) = \left\{ \begin{array}{rcl} 8 &\mathrm{if}& \iota \not\in\inner{\iota_1}\\ 16 &\mathrm{if}& \iota = \iota_1\\\end{array}\right.\]
while at any of the other fixed points $y$, one has
	\[\cod(M^\iota_y) = \left\{ \begin{array}{rcl}
	4   & \mathrm{if}& \iota = \iota_1\\
	8   & \mathrm{if}& \iota \not\in \iota_1 C(y)\\
	12 & \mathrm{if}& \iota \in \iota_1 C(y) \setminus \{\iota_1\}
	\end{array}\right.\]
for some subgroup $C(y)$ isomorphic to $\Z_2^2$ depending on $y$. For example, $C(y) = \inner{\iota_3,\iota_4}$ at $y = [0,1,0,0,0]$. Notice also the fixed point component $M^\iota_y$ is $\C\pp^4$ for $\iota \in \iota_2 C(y) \cup \iota_1\iota_2 C(y)$ and is $\HH\pp^2$ for non-trivial $\iota \in C(y)$. Finally, note that the fixed point set of the $T^4$ action is contained in the two-fold union of a large number of choices of eight-dimensional fixed point components of involutions. We recover this combinatorial and topological fixed point data in one case of the proof of Lemma \ref{lem:dim16codim4}, however we are unable to fully recover the topology of $M$.
\end{example}

\begin{proof}[Proof of Lemma \ref{lem:dim16codim4}]
By the codimension two lemma, we may assume that $\cod\of{M^{\iota_1}} = 4$. Choose $x \in M^T$ such that $\cod\of{M^{\iota_1}_x} = 4$. By the codimension four lemmas, we may assume that every other non-trivial involution $\iota \in T$ has $\cod\of{M^{\iota}} \geq 8$ and that $\cod(M^{\iota}_z) = 8$ for any $z \in M^T \setminus M^{\iota_1}_x$ and $\iota \in \Z_2^4 \setminus\inner{\iota_1}$.

Let $y \in M^T \cap M^{\iota_1}_x$. The isotropy map at $y$ is also rigid in the sense that there exist a rank-two subgroup $C(y)\subseteq \Z_2^4$ and an involution $\iota_2(y) \not\in C(y)$ such that $\cod\of{M^\iota_y} = 8$ for each of the $11$ non-trivial involutions $\iota$ not in the coset $\iota_1 C(y)$. Moreover, if $\iota$ is one of these $11$ involutions, the intersection $M^{\iota_1}_x \cap M^\iota_y$ is transverse if and only if $\iota \in C(y)$.

Fix $\iota_2$ to be any choice of $\iota_2(x)$, and fix any choice of distinct $\iota_3,\iota_4 \in C(x)$. From this isotropy rigidity and Frankel's theorem, all of the following hold:
	\begin{enumerate}
	\item $\cod\of{M^\iota_x} = 8$ and $M^T \subset M^{\iota_1}_x \cup M^\iota_x$ for all $\iota \in \{\iota_2, \iota_3, \iota_4, \iota_3\iota_4\}$.
	\item $M^{\iota_1}_x \cap M^{\iota_2}_x \subset M^{\iota_2}_x$ is $4$--connected and has codimension two.
	\item $M^{\iota_1}_x \cap M^{\iota}_x \subset M^{\iota}_x$ is $4$--connected and has codimension four for $\iota \in \{\iota_3, \iota_4, \iota_3\iota_4\}$.
	\item $M^T \subseteq M^{\iota_3}_x \cup M^{\iota_4}_x \cup M^{\iota_3\iota_4}_x$.
	\end{enumerate}

Set $\ep = \chi(M) - \chi(M^{\iota_1}_x)$. By the first claim above, $\ep = \chi(M^{\iota_2}_x) - \chi(M^{\iota_1}_x \cap M^{\iota_2}_x)$. By the second claim and the codimension two lemma, $M^{\iota_2}_x$ and $M^{\iota_1}_x \cap M^{\iota_2}_x$ are cohomology spheres or complex projective spaces. In either case, the difference in their Euler characteristics is $\ep \in \{0, 1\}$, and $\chi\of{M^{\iota_2}_x} = 2 + 3\ep$. Note that the fixed points of $T$ are isolated in this case, so
	\[\chi(M) \geq \chi\of{M^{\iota_2}_x} \geq 2 + 3\ep.\]

We prove now that the opposite inequality holds. Fix for a moment $\iota \in \{\iota_3, \iota_4, \iota_3\iota_4\}$. By the first claim, $\ep = \chi(M^\iota_x) - \chi(M^{\iota_1}_x \cap M^\iota_x)$.  By the third claim, $M^{\iota}_x$ has $4$--periodic integral cohomology and second Betti number equal to that of $M^{\iota_1}_x \cap M^\iota_x$. In particular, since $\ep \in \{0,1\}$, we have $\chi\of{M^\iota_x} = 2 + \ep$ for all $\iota \in \{\iota_3, \iota_4, \iota_3\iota_4\}$. Now consider any two-fold intersection $M^\iota_x \cap M^{\iota'}_x$ of these three submanifolds. Since the fixed points of $T$ are isolated, the Euler characteristic of $M^\iota_x \cap M^{\iota'}_x$ is at least that of its component $M^{\inner{\iota,\iota'}}_x$. This component is a closed, oriented, positively curved $4$--manifold, so its Euler characteristic is at least two. In addition, this Euler characteristic is at least that of the three-fold intersection $M^{\iota_3}_x \cap M^{\iota_4}_x \cap M^{\iota_3\iota_4}_x$. Putting these estimates together, the fourth claim above and the inclusion-exclusion property of Euler characteristics implies
	\[\chi(M) \leq 3(2+ \ep) - 2 - 2 - 0 = 2 + 3 \ep.\]
Since the opposite inequality also holds, we conclude $\chi(M) = 2 + 3\ep \in \{2, 5\}$.

For the signature, recall that $M^{\iota_2}_x$ has $2$--periodic cohomology and is either $\s^8$ or $\C\pp^4$, according to whether $\ep$ is $0$ or $1$. In either case $\chi(M) = \chi\of{M^{\iota_2}_x}$, so $M^{\iota_2}$ is connected since the fixed points of $T$ are isolated. In particular, $\sigma(M) = \sigma\of{M^{\iota_2}_x}$, which is either $0$ or $\pm 1$, according to whether $\chi(M)$ is $2$ or $5$, respectively.

We conclude the rest of the lemma by considering cases. Let $N_1 = M^{\iota_1}_x$ and recall that $\dk(N_1)$ denotes the dimension of the kernel of the induced $T$--action on $N_1$.
	\begin{itemize}
	\item If $\dk(N_1) = 0$, Wilking's homotopy classification implies that $N_1$ is a cohomology sphere or quaternionic projective space. By the connectedness lemma, $M$ is as well.
	\item If $\dk(N_1) = 1$ and $\chi(M) = 2$, then $\chi(N_1) = 2$ as well. Since the fixed points of $T$ are isolated, $M^{\iota_1}$ is connected and hence $H^i(M;\Z_2) = 0$ for $3 \leq i \leq n-3$. Since $\chi(M) = 2 + 2b_2(M)$, we have $b_2(M) = 0$ and hence that $M$ is homeomorphic to $\s^{16}$ by Lemma \ref{lem:cod4dk1Qsphere}.
	\item If $\dk(N_1) = 1$ and $\chi(M) = 5$, the proof shows that $M^{\iota_3}_x$ is homeomorphic to $\HH\pp^2$. By Wilking's maximal smooth symmetry rank classification for a integral quaternionic projective spaces (see \cite[Theorem 3]{Wilking03}), $M^{\iota_3}_x$ is fixed by a circle in $T$. In particular, $M^{\iota_3}_x \to M$ is $2$--connected, so $H_2(M;\Z) = 0$. Wilking's periodicity corollary now implies that $H^{2+4i}(M;\Z) = 0$ and $H^4(M;\Z) \supseteq \Z$.
	\end{itemize}
\end{proof}

We now consider the possibility that the minimal codimension of a fixed-point set of an involution in $T$ is six.

\begin{lemma}\label{lem:dim16codim6}
Suppose $\min \cod\of{M^\iota} = 6$, where the minimum runs over involutions in $T$. One of the following occurs:
	\begin{itemize}
	\item $\chi(M) = 2$, $\sigma(M) = 0$, and $M$ is $5$--connected.
	\item $\chi(M) = 9$, $\sigma(M) = \pm 1$, $H_2(M;\Z) = \Z$, $H_3(M;\Z) = 0$, $H_4(M;\Z) = \Z$, $H_5(M;\Z) = 0$, and $H^{16}(M;\Z)$ is generated by an element of the form $x^5 y$ with $x \in H^2(M;\Z)$ and $y \in H^6(M;\Z)$, and $M$ is not spin.
	\end{itemize}
\end{lemma}

\begin{proof}
Choose $x \in M^T$ such that $N_1 = M^{\iota_1}_x$ has codimension six. The isotropy at $x$ implies that some $\iota_2 \in \Z_2^4 \setminus\langle\iota_1\rangle$ exists such that $M^{\iota_2}_x$ has codimension six. Moreover, if $N_1$ and $N_2$ intersect transversely, there exists $\iota_3\in\Z_2^4 \setminus\langle\iota_1,\iota_2\rangle$ such that $M^{\iota_3}_x$ has codimension six and intersects $N_1$ non-transversely. We may assume therefore that $N_2 = M^{\iota_2}_x$ has codimension six and intersects $N_1$ non-transversely.

By the containment lemma,
	\[M^T \subseteq N_1 \cup N_2 \cup N_{12},\]
where $N_{12} = M^{\iota_1\iota_2}_x$. By the connectedness lemma, all two-fold intersections of $N_1$, $N_2$, and $N_{12}$ are connected and hence equal $N_1 \cap N_2 = M^{\inner{\iota_1,\iota_2}}_x$. Since the intersection $N_1 \cap N_2$ is not transverse, $N_1 \cap N_2 \subseteq N_{12}$ has codimension two. Since no involution $\iota$ has $\cod\of{M^\iota} < 6$, $N_{12}$ has $T^3$ symmetry, so $N_{12}$ and $N_1 \cap N_2$ are both cohomology spheres or both cohomology complex projective spaces. We consider two cases:
	\begin{enumerate}
	\item Suppose $N_1 \cap N_2$ is homeomorphic to $\s^6$. Since $N_1 \cap N_2 \to N_i$ is $4$--connected for $i \in \{1,2\}$, each $N_i$ is a $4$--connected $10$--manifold. Since $N_i$ is positively curved and has $T^3$ symmetry, $\chi(N_i) > 0$, which implies that $N_i$ is homeomorphic to $\s^{10}$. By the inclusion-exclusion property of the Euler characteristic, $\chi(M) = 2$. Moreover, since $N_1 \to M$ is $5$--connected, $M$ is $5$--connected as well. For the signature, note that $\sigma(M) = \sigma(M^{\iota_1}) = \sigma(N_1) = 0$.
	\item Suppose $N_1 \cap N_2$ is a cohomology $\C\pp^3$. Since $N_1 \cap N_2 \to N_i$ is $4$--connected for $i \in \{1,2\}$, it follows by arguments similar to those in the dimension $10$ result that each $N_i$ is a cohomology $\C\pp^5$. It follows that $\chi(M) = 9$, $H_2(M;\Z) = \Z$, $H_3(M;\Z) = 0$, $H_4(M;\Z) = \Z$, $H_5(M;\Z) = 0$, and some $x \in H^2(M;\Z)$ satisfies the property that $x^5$ is not a multiple. In particular, $H^{16}(M;\Z)$ is generated by an element of the form $x^5 y$ with $y \in H^6(M;\Z)$.
	
	For the signature, note that $\sigma(M) = \pm 1$ follows immediately from the formula $\sigma(M) = \sigma(M^{\iota_1})$ unless $M^{\iota_1}$ consists of $N_1$ together with three isolated points $y_j$. In this latter case, no $y_j \in N_{12}$, as that would imply that $\cod(M^{\iota_2}_{y_j})  = 8$, a contradiction to Frankel's theorem since $\cod\of{M^{\iota_2}} = 6$. It follows that each $y_j \in N_2$ and hence that each $y_j$ lies in a $6$-dimensional component of $M^{\iota_1\iota_2}$. Since $\chi(N_{12}) = 6$, the fixed-point set of $\iota_1\iota_2$ is comprised of $N_{12}$ together with some number of six-dimensional components whose Euler characteristic is at least, and hence sums to, three. This contradicts the fact that closed, oriented six-manifolds have even Euler characteristics.
	
	Finally, if $M$ is spin, then $M^{\iota_1}$ has, in addition to $N_1^{10}$, components of dimension two. These two-dimensional components are oriented, so they are diffeomorphic to spheres. This is a contradiction since $\chi(M) - \chi(N_1) = 3$.
	\end{enumerate}
\end{proof}

To complete the proof of Theorem \ref{thm:dim16}, we prove the following:

\begin{lemma}[Rigid isotropy lemma for $n = 16$]\label{lem:dim16codim8}
If $\cod\of{M^\iota} \geq 8$ for every non-trivial involution $\iota \in T$, then one of the following occurs:
	\begin{itemize}
	\item $\chi(M) = 2$ and $\sigma(M) = 0$.
	\item $\chi(M) = 3$, $\sigma(M) = \pm 1$, and $M$ is $2$--connected.
	\end{itemize}
\end{lemma}

\begin{proof}
The isotropy representation is rigid in this case. Indeed, for each $x \in M^T$, there exists an involution $\iota_0$ such that $\cod\of{M^{\iota_0}_x} = 16$ and $\cod\of{M^\iota_x} = 8$ for every other non-trivial involutions $\iota$.

Fix $x \in M^T$ and choose distinct, non-trivial involutions $\iota_1,\iota_2 \in T$ such that each $M^{\iota_i}_x$ has codimension eight. By the rigidity of the isotropy, each $N_i$ has $T^3$ symmetry, so $\chi(N_i) \leq \chi(\C\pp^4) = 5$ by Fang and Rong's homeomorphism classification in dimension eight. By Frankel's theorem, $M^T \subseteq M^{\iota_1}_x \cup M^{\iota_2}_x$, so we have $\chi(M) \leq 5 + 5 < 14$.

Since there are fourteen involutions $\iota$ with $\cod\of{M^\iota_x} = 8$, and since at most one of these has $\cod\of{M^\iota_y} \neq 8$ at any other fixed point $y \in M^T$, not every involution gets a turn at having a $0$--dimensional fixed point component. By Frankel's theorem, there exists an involution $\iota \in T^4$ such that $M^\iota$ is connected and has dimension eight. By the rigidity of the isotropy representation, $M^\iota$ has $T^3$ symmetry and hence is homeomorphic to $\s^8$, $\HH\pp^2$, or $\C\pp^4$ by the Fang--Rong classification. We consider these three cases separately.
	\begin{enumerate}
	\item $M^\iota = \s^8$. The Euler characteristic and signature of $M$ and $M^\iota$ are the same, so the calculation $\chi(M) = 2$ and $\sigma(M) = 0$ follows immediately. This is the first possible conclusion of Lemma \ref{lem:dim16codim8}.
	\item $M^\iota = \HH\pp^2$. As in the previous case, we immediately conclude $\chi(M) = 3$ and $\sigma(M) = \pm 1$. Moreover, we conclude by Wilking's maximal smooth symmetry rank bound for $\HH\pp^2$ that there is a circle in $T^4$ fixing $M^\iota$. The connectedness lemma implies that $M^\iota \to M$ is $2$--connected, so we have the additional conclusion in this case that $\pi_2(M) = 0$.
	\item $M^\iota = \C\pp^4$. We claim this case cannot occur. Indeed, choose another non-trivial involution $\tau$, and consider the induced action of $\tau$ on $M^\iota$. Since $M^\iota$ is homeomorphic to $\C\pp^4$, the fixed point set $(M^\iota)^\tau$ is comprised of exactly two components, and these two components are integral complex projective spaces whose dimensions add to $\dim(\C\pp^4) - 2 = 6$ in accordance with the condition that $\chi(M^\iota) = \chi((M^\iota)^\tau)$. In particular, one of the components of $(M^\iota)^\tau$ has dimension two or six. This contradicts the rigidity of the isotropy representation, which implies that every component of $M^{\langle\iota,\tau\rangle}$ has dimension zero or four.
	\end{enumerate}
This concludes the proof of Lemma \ref{lem:dim16codim8} and hence of Theorem \ref{thm:dim16}.
\end{proof}

\smallskip\section{Elliptic genus calculation}\label{sec:EllipticGenus}\smallskip

We present a unified proof of the elliptic genus claims in dimensions $12$ and $16$.

\begin{theorem}
Let $M^{4m}$ be a closed, simply connected, spin manifold admitting positive sectional curvature and $T^m$ symmetry. If $m \leq 4$, the elliptic genus is constant.
\end{theorem}

Note that, if $T^m$ is replaced by $T^{m+1}$, then this claim holds by Wilking \cite[Theorem 2]{Wilking03}. Indeed, restricting his result to the spin case, such a manifold is homeomorphic to $\s^{4m}$ or $\HH\pp^m$, and so the elliptic genus is constant by Novikov's theorem.

\begin{proof}
Since $M$ is spin, every involution $\iota \in T^m$ has the property that there exists $c \in \{0,2\}$ such that $\cod\of{N} \equiv c \bmod{4}$ for all components $N \subseteq M^\iota$. If $c = 2$ for some involution, the action of that involution is of odd type and it follows that the elliptic genus vanishes by a result of Hirzebruch--Slodowy (see corollary on \cite[page 317]{HS90}). We assume therefore that every component of the fixed-point set of every involution has codimension divisible by four.

If there exists an involution $\iota \in T^m$ for which $\cod\of{M^\iota} \geq \frac 1 2 \dim(M)$, then the elliptic genus is constant by another result of Hirzebruch--Slodowy (see the above-cited corollary). In particular, we are done if $m \leq 2$. Moreover, we are done $m \in \{3,4\}$ and there exists an involution $\iota \in T^m$ and a component $N \subseteq M^\iota$ satisfying $\cod\of{N} = 8$. Indeed, Frankel's theorem implies that any other such component has codimension at least $8$, which is at least half of the dimension of $M$. Assume therefore that no component of a fixed-point set of an involution has codimension eight.

Next, suppose there exist two involutions with fixed point components, $N_1$ and $N_2$, of codimension four. If $N_1 \cap N_2$ is transverse, then the product of these involutions has fixed-point set of codimension eight, a contradiction. If, on the other hand, $N_1 \cap N_2$ is not transverse, then the codimension two lemma implies that $N_1$ is homotopy equivalent to a sphere or complex projective space, and the codimension four lemma implies that $M$ is an integral sphere, complex projective space, or quaternionic projective space. Since $N_1 \to M$ is $5$--connected and $M$ is spin, it follows that $M$ is homeomorphic to $\s^{4m}$ and hence that the elliptic genus is constant.

Assume therefore that at most one involution has codimension four fixed-point set. Choose $\Z_2^{m-1} \subseteq \Z_2^m$ such that every $\iota \in \Z_2^{m-1}$ has $\cod\of{M^{\iota}} \geq 12$. Summing these codimensions at some fixed point $x \in M^T$, we have
	\[\of{2^{m-1} - 1}(12) \leq \sum_{\iota \in \Z_2^{m-1}} \cod\of{M^\iota_x} = 2^{m-2} \cod\of{M^{\Z_2^{m-1}}_x} \leq 2^m m.\]
Since $m \in \{3,4\}$, this is a contradiction, so the proof is complete.
\end{proof}

\smallskip\section{Low dimensional positively elliptic spaces}\label{sec:F0spaces}\smallskip

A simply connected, rationally elliptic topological space with positive Euler characteristic is called an \emph{$F_0$}, or \emph{positively elliptic}, space. These spaces have a (formal) dimension and satisfy Poincar\'e duality, and they admit pure minimal models (see \cite{FelixHalperinThomas01}). Specifically, they admit minimal models $(\Lambda V, d)$ with generators $x_1,\ldots,x_k \in V$ of even degree and generators $y_1,\ldots,y_k \in V$ of odd degree such that each $\dif x_i = 0$ and each $\dif y_i$ is a homogeneous polynomial in the $x_1,\ldots,x_k$. We classify all possible tuples
	\[(\deg x_1,\ldots,\deg x_k, \deg y_1, \ldots, \deg y_k)\]
of degrees for $F_0$ spaces of dimension up to $16$.

In dimensions up to eight, this follows from previous work (see \cite{PaternainPetean03,Pavlov02,Herrmann-pre}).

\begin{theorem}[Paternain--Petean, Pavlov, Herrmann]
If $M$ is an $F_0$ space of formal dimension $2$, $4$, $6$, or $8$, then $M$ admits a pure model whose tuples of homotopy generator degrees satisfy one of the following:
	\begin{itemize}
	\item $\dim M = 2$ and the tuple of degrees is $(2,3)$, and $M \simeq_\Q \s^2$
	\item $\dim M = 4$ and the tuple of degrees is $(2,5)$, $(4,7)$, or $(2,2,3,3)$, and $M$ is rationally homotopy equivalent to $\s^4$, $\C\pp^2$, $\s^2 \times \s^2$, $\C\pp^2\#\C\pp^2$, or $\C\pp^2\#-\C\pp^2$.
	\item $\dim M = 6$ and the tuple of degrees appears in Table \ref{tab6}.
	\item $\dim M = 8$ and the tuple of degrees appears in Table \ref{tab8}.
	\end{itemize}
\end{theorem}

\begin{longtable}{c|c}
\caption{Dimension $6$}\\
\label{tab6}
$\deg \pi_*(M)\otimes \qq$ & $\chi(M)$ \\
\hline
$(6,11)$ & $2$ \\
$(2,7)$, $(2,4,3,7)$ & $4$\\
$(2,2,3,5)$ & $6$\\
$(2,2,2,3,3,3)$ & $8$
\end{longtable}

\begin{longtable}{c|c}
\caption{Dimension $8$}\\
\label{tab8}
$\deg \pi_*(M)\otimes \qq$ & $\chi(M)$ \\
\hline
$(8,15)$ 				& $2$\\
$(4,11)$ 				&$3$\\
$(2,6,3,11)$, $(4,4,7,7)$ 	&$4$\\
$(2,9)$				&$5$\\
$(2,4,5,7)$ 			&$6$\\
$(2,2,3,7)$, $(2,2,4,3,3,7)$ & $8$\\
$(2,2,5,5)$ 			&$9$\\
$(2,2,2,3,3,5)$ 			&$12$\\
$(2,2,2,2,3,3,3,3)$ 		&$16$
\end{longtable}

In dimensions six and eight, Herrmann \cite{Herrmann-pre} proves a partial classification of the rational homotopy types in dimensions, but we will not need this here. We proceed to the computation in dimensions $10$, $12$, $14$, and $16$. Let $M$ be an $F_0$-space.

\begin{theorem}\label{thm:PossibleTuples}
If $M$ is an $F_0$ space of formal dimension $10$, $12$, $14$, or $16$, then $M$ admits a pure model whose tuples of homotopy generator degrees appear in one of the Tables \ref{tab10}, \ref{tab12}, \ref{tab14}, or \ref{tab16}. Moreover, each tuple that appears in this table is realized by such a space.
\end{theorem}

\begin{proof}
Let $(\Lambda V, d)$ be a pure model of $M$ with $k$ generators $x_i$ of even degree $2a_i$ and $k$ generators $y_i$ of odd degree $2b_i-1$. Following \cite[Section 32]{FelixHalperinThomas01}, we may choose such $x_i$ and $y_i$ such that the following hold:

\begin{itemize}
\item $1 \leq a_1\leq \ldots \leq a_k$ and $2 \leq b_1\leq \ldots \leq b_k$.
\item $b_i\geq 2 a_i$ for all $1\leq i\leq k$.
\item $2\sum_{i = 1}^k (b_i - a_i) = \dim M$.
\end{itemize}

Note in particular that $\dim(M) \geq 2 \sum a_i \geq k$, so in each dimension there are finitely many possible values for $k$ and for the $a_i$. It follows that there are only finitely many possible values for the $b_i$ as well. We lead a computer based search using Mathematica to enumerate all possible tuples $(2a_1,\ldots,2a_k, 2b_1-1,\ldots,2b_k-1)$ where $k$ ranges from one up to half the dimension of $M$. Moreover, we compute in each case the Euler characteristic, which satisfies the following formula:
	\[\chi(M) = \prod_{i=1}^k \frac{\deg y_i + 1}{\deg x_i} = \prod_{i=1}^k \frac{b_i}{a_i}.\]
Of course, we can rule out any tuples for which the Euler characteristic is non-integral. However, there is a single criterion called the ``arithmetic condition'' due to Friedlander and Halperin that characterizes precisely which tuples of degrees are realized by $F_0$ spaces (see \cite[Proposition 32.9]{FelixHalperinThomas01}). In practice, it is not too difficult to check whether any given tuple of degrees can be realized using the integrality of the Euler characteristic together with direct arguments.

Let us demonstrate such a direct argument in one exemplary case, in which we exclude the existence of such a model. Suppose that homotopy groups are given degrees $(2,4,4,5,7,9)$. That is, $V=\langle x_1,x_2,x_3,y_1,y_2,y_3\rangle$ with $\deg x_1=2, \deg x_2=4, \deg x_3=4$, $\deg y_1=5, \deg y_2=7, \deg y_3=9$. Let us see that there is no differential $\dif$ on $\Lambda V$ making $H(\Lambda V,\dif)$ finite dimensional. In fact, for degree reasons we see that $\dif y_1\in x_1\cdot \langle x_1^2,x_2, x_3\rangle$ and $\dif y_3\in x_1\cdot\langle x_1^4,x_2x_3,x_2^2,x_3^2\rangle$. Thus these two relations taken together cannot reduce the Krull dimension of $\qq[x_1,x_2,x_3]$ by two.

For all the configurations in the tables it is easy to construct minimal Sullivan models. (Again, one can alternatively check that the arithmetic condition of Friedlander and Halperin holds.) Indeed, nearly all of them can be realized by products of spaces with singly-generated cohomology algebra. Let us also provide the arguments in a few cases where we need slightly more complicated realizing models:

The first configuration which is not realizable as a non-trivial rational product is given by $(4,6,9,11)$. Here we may consider
\begin{align*}
(\Lambda \langle x_1,x_2, y_1,y_2\rangle,\dif)
\end{align*}
with $\deg x_1=4, \deg x_2=6$, $\deg y_1=9, \deg y_2=11$ and $\dif x_1=\dif x_2=0$, $\dif y_1=x_1x_2, \dif y_2=x_1^3+x_2^2$.

The next case is $(2,4,6,3,9,11)$, which can be realized by the product of the model above and the one of the $2$-sphere.

The case $(4,6,11,13)$ is realized by
\begin{align*}
(\Lambda \langle x_1,x_2, y_1,y_2\rangle,\dif)
\end{align*}
with $\deg x_1=4, \deg x_2=6$, $\deg y_1=11, \deg y_2=13$ and $\dif x_1=\dif x_2=0$, $\dif y_1=x_1^3+x_2^2, \dif y_2=x_1^2x_2$.

The next case is $(2,4,6,5,9,11)$, which can be realized by the product of the model above and the one of $\cc\pp^2$. In the case $(2, 2, 4, 6, 3, 3, 9, 11)$, we use two sphere factors instead. For $(2, 2, 4, 6, 3, 3, 9, 11)$, we use two sphere factors.

The configuration $(4,4,6,7,9,11)$ can be realized by
\begin{align*}
(\Lambda \langle x_1,x_2,x_3,y_1,y_2,y_3\rangle,\dif)
\end{align*}
with $\deg x_1=\deg x_2=4$, $\deg x_3=6$, $\deg y_1=7, \deg y_2=9, \deg y_3=11$ and $\dif x_1=\dif x_2=\dif x_3=0$, $\dif y_1=x_1^2, \dif y_2=x_2x_3, \dif y_3=x_2^3+x_2^2$.
\end{proof}

\begin{longtable}{c|c}
\caption{Dimension $10$}\\
\label{tab10}
$\deg \pi_*(M)\otimes \qq$ & $\chi(M)$ \\
\hline
$(10,19)$ & $2$\\
$(2,8,3,15)$, $(4,6,7,11)$ & $4$\\
$(2,11)$, $(2,4,3,11)$, $(2,6,5,11)$& $6$\\
$(2,4,7,7)$, $(2,2,6,3,3,11)$, $(2,4,4,3,7,7)$&$8$\\
$(2,2,3,9)$ &$10$\\
$(2,2,5,7)$, $(2,2,4,3,5,7)$ &$12$\\
$(2,2,2,3,3,7)$, $(2,2,2,4,3,3,3,7)$&$16$\\
$(2,2,2,3,5,5)$&$18$\\
$(2,2,2,2,3,3,3,5)$&$24$\\
$(2,2,2,2,2,3,3,3,3,3)$&$32$
\end{longtable}

\begin{longtable}{c|c}
\caption{Dimension $12$}\\
\label{tab12}
$\deg \pi_*(M)\otimes \qq$ & $\chi(M)$ \\
\hline
$(12,23)$ & $2$\\
$(6,17)$ & $3$\\
$(4,15)$, $(2,10,3,19)$, $(4,8,7,15)$, $(6,6,11,11)$ & $4$\\
$(4,6,9,11)$&$5$\\
$(2,8,5,15)$, $(4,4,7,11)$& $6$\\
$(2,13)$ & $7$\\
$(2,6,7,11)$, $(2,2,8,3,3,15)$, $(2,4,6,3,7,11)$, $(4,4,4,7,7,7)$&$8$\\
$(2,4,5,11)$&$9$\\
$(2,4,7,9)$&$10$\\
$(2,2,3,11)$, $(2,2,4,3,3,11)$, $(2,2,6,3,5,11)$, $(2,4,4,5,7,7)$&$12$\\
$(2,2,5,9)$ &$15$\\
$(2,2,7,7)$, $(2,2,4,3,7,7)$, $(2,2,2,6,3,3,3,11)$, $(2,2,4,4,3,3,7,7)$&$16$\\
$(2,2,4,5,5,7)$&$18$\\
$(2,2,2,3,3,9)$&$20$\\
$(2,2,2,3,5,7)$, $(2,2,2,4,3,3,5,7)$&$24$\\
$(2,2,2,5,5,5)$&$27$\\
$(2,2,2,2,3,3,3,7)$, $(2,2,2,2,4,3,3,3,3,7)$&$32$\\
$(2,2,2,2,3,3,5,5)$&$36$\\
$(2,2,2,2,2,3,3,3,3,5)$&$48$\\
$(2,2,2,2,2,2,3,3,3,3,3,3)$&$64$
\end{longtable}

\begin{longtable}{c|c}
 \caption{Dimension $14$}\\
\label{tab14}
$\deg \pi_*(M)\otimes \qq$ & $\chi(M)$ \\
\hline
$(14,27)$ & $2$\\
$(2,12,3,23)$, $(4,10,7,19)$, $(6,8,11,15)$& $4$\\
$(2,6,3,17)$, $(2,10,5,19)$, $(4,6,11,11)$& $6$\\
$(2,15)$, $(2,4,3,15)$, $(2,8,7,15)$, $(2,2,10,3,3,19)$,
& \\ $(2,4,8,3,7,15)$, $(2,6,6,3,11,11)$, $(4,4,6,7,7,11)$&$8$\\
$(2,6,9,11)$, $(2,4,6,3,9,11)$&$10$\\
$(2,4,7,11)$, $(2,2,8,3,5,15)$, $(2,4,4,3,7,11)$, $(2,4,6,5,7,11)$&$12$\\
$(2,2,3,13)$ &$14$\\
$(2,2,6,3,7,11)$, $(2,4,4,7,7,7)$, $(2,2,2,8,3,3,3,15)$,
&\\ $(2,2,4,6,3,3,7,11)$, $(2,4,4,4,3,7,7,7)$&$16$\\
$(2,2,5,11)$, $(2,2,4,3,5,11)$, $(2,2,6,5,5,11)$&$18$\\
$(2,2,7,9)$, $(2,2,4,3,7,9)$&$20$\\
$(2,2,2,3,3,11)$, $(2,2,4,5,7,7)$, $(2,2,2,4,3,3,3,11)$,
&\\ $(2,2,2,6,3,3,5,11)$, $(2,2,4,4,3,5,7,7)$&$24$\\
$(2,2,2,3,5,9)$&$30$\\
$(2,2,2,3,7,7)$, $(2,2,2,4,3,3,7,7)$, $(2,2,2,2,6,3,3,3,3,11)$,
&\\ $(2,2,2,4,4,3,3,3,7,7)$&$32$\\
$(2,2,2,5,5,7)$, $(2,2,2,4,3,5,5,7)$&$36$\\
$(2,2,2,2,3,3,3,9)$&$40$\\
$(2,2,2,2,3,3,5,7)$, $(2,2,2,2,4,3,3,3,5,7)$&$48$\\
$(2,2,2,2,3,5,5,5)$&$54$\\
$(2,2,2,2,2,3,3,3,3,7)$, $(2,2,2,2,2,4,3,3,3,3,3,7)$&$64$\\
$(2,2,2,2,2,3,3,3,5,5)$&$72$\\
$(2,2,2,2,2,2,3,3,3,3,3,5)$&$96$\\
$(2,2,2,2,2,2,2,3,3,3,3,3,3,3)$&$128$
\end{longtable}

\begin{longtable}{c|c}
\caption{Dimension $16$}\\
\label{tab16}
$\deg \pi_*(M)\otimes \qq$ & $\chi(M)$ \\
\hline
$(16,31)$&$2$\\
$(8,23)$&$3$\\
$(2,14,3,27)$, $(4,12,7,23)$, $(6,10,11,19)$, $(8,8,15,15)$&$4$\\
$(4,19)$&$5$\\
$(2,12,5,23)$, $(4,6,7,17)$, $(4,8,11,15)$&$6$\\
$(4,6,11,13)$&$7$\\
$(2,10,7,19)$, $(4,4,7,15)$, $(2,2,12,3,3,23)$, $(2,4,10,3,7,19)$,
&\\$(2,6,8,3,11,15)$,
$(4,4,8,7,7,15)$, $(4,6,6,7,11,11)$&$8$\\
$(2,17)$, $(2,6,5,17)$, $(4,4,11,11)$&$9$\\
$(2,8,9,15)$, $(4,4,6,7,9,11)$&$10$\\

$(2,4,5,15)$, $(2,6,11,11)$, $(2,2,6,3,3,17)$,$(2,2,10,3,5,19)$,
&\\$(2,4,6,3,11,11)$,$(2,4,8,5,7,15)$, $(2,6,6,5,11,11)$, $(4,4,4,7,7,11)$&$12$\\

$(2,4,7,13)$&$14$\\

$(2,4,9,11)$, $(2,4,6,5,9,11)$&$15$\\

$(2,2,3,15)$, $(2,2,4,3,3,15)$, $(2,2,8,3,7,15)$,
&\\$(2,4,6,7,7,11)$, $(2,2,2,10,3,3,3,19)$, $(2,2,4,8,3,3,7,15)$,
&\\$(2,2,6,6,3,3,11,11)$, $(2,4,4,6,3,7,7,11)$, $(4,4,4,4,7,7,7,7)$&$16$\\

$(2,2,8,5,5,15)$, $(2,4,4,5,7,11)$&$18$\\

$(2,2,6,3,9,11)$, $(2,4,4,7,7,9)$, $(2,2,4,6,3,3,9,11)$&$20$\\

$(2,2,5,13)$&$21$\\

$(2,2,7,11)$, $(2,2,4,3,7,11)$, $(2,2,6,5,7,11)$, $(2,2,2,8,3,3,5,15)$,
&\\$(2,2,4,4,3,3,7,11)$, $(2,2,4,6,3,5,7,11)$, $(2,4,4,4,5,7,7,7)$&$24$\\

$(2,2,9,9)$&$25$\\
$(2,2,4,5,5,11)$&$27$\\
$(2,2,2,3,3,13)$&$28$\\
$(2,2,4,5,7,9)$,
&$30$\\

$(2,2,4,7,7,7)$, $(2,2,2,6,3,3,7,11)$, $(2,2,4,4,3,7,7,7)$,
&\\$(2,2,2,2,8,3,3,3,3,15)$,
$(2,2,2,4,6,3,3,3,7,11)$, $(2,2,4,4,4,3,3,7,7,7)$&$32$\\

$(2,2,2,3,5,11)$, $(2,2,2,4,3,3,5,11)$, $(2,2,2,6,3,5,5,11)$,
&\\$(2,2,4,4,5,5,7,7)$&$36$\\
$(2,2,2,3,7,9)$, $(2,2,2,4,3,3,7,9)$&$40$\\
$(2,2,2,5,5,9)$&$45$\\

$(2,2,2,5,7,7)$, $(2,2,2,2,3,3,3,11)$, $(2,2,2,4,3,5,7,7)$,
&\\$(2,2,2,2,4,3,3,3,3,11)$, $(2,2,2,2,6,3,3,3,5,11)$, $(2,2,2,4,4,3,3,5,7,7)$&$48$\\

$(2,2,2,4,5,5,5,7)$&$54$\\
$(2,2,2,2,3,3,5,9)$&$60$\\
$(2,2,2,2,3,3,7,7)$, $(2,2,2,2,4,3,3,3,7,7)$,
&\\$(2,2,2,2,2,6,3,3,3,3,3,11)$, $(2,2,2,2,4,4,3,3,3,3,7,7)$&$64$\\
$(2,2,2,2,3,5,5,7)$, $(2,2,2,2,4,3,3,5,5,7)$&$72$\\
$(2,2,2,2,2,3,3,3,3,9)$&$80$\\
$(2,2,2,2,5,5,5,5)$&$81$\\
$(2,2,2,2,2,3,3,3,5,7)$, $(2,2,2,2,2,4,3,3,3,3,5,7)$&$96$\\
$(2,2,2,2,2,3,3,5,5,5)$&$108$\\
$(2,2,2,2,2,2,3,3,3,3,3,7)$, $(2,2,2,2,2,2,4,3,3,3,3,3,3,7)$&$128$\\
$(2,2,2,2,2,2,3,3,3,3,5,5)$&$144$\\
$(2,2,2,2,2,2,2,3,3,3,3,3,3,5)$&$192$\\
$(2,2,2,2,2,2,2,2,3,3,3,3,3,3,3,3)$&$256$

\end{longtable}

\smallskip\section{The Halperin conjecture}\label{sec:Halperin}\smallskip

In this section we will show that $F_0$-spaces of dimension at most $16$ satisfy the Halperin conjecture; and so do $F_0$-spaces with Euler characteristic at most $16$. Recall that the Halperin conjecture states that any fibration with an $F_0$-space as a fiber should yield a Leray--Serre spectral sequence degenerating at the $E_2$-term.

Keeping the notation from the previous section, let $M$ be an $F_0$ space of dimension $\dim M \leq 16$, and let $(\Lambda(x_1,\ldots,x_k,y_1,\ldots,y_k),\dif)$ be a pure model satisfying the following properties:
	\begin{itemize}
	\item The degrees $\deg(x_i)$ are increasing, and likewise for the degrees $\deg(y_i)$.
	\item $\deg(\dif y_i) \geq 2 \deg(x_i)$ for all $i$.
        	\item each $\dif x_i = 0$ and each $\dif y_i$ is a homogeneous polynomial in $x_1,\ldots,x_k$.
	\end{itemize}

Recall that the Halperin conjecture is known in the following cases:
\begin{itemize}
\item $k \leq 3$, i.e., if the cohomology algebra of $M$ is generated by at most three elements (see Lupton \cite{Lupton90}).
\item $\deg(x_1) = \ldots = \deg(x_k)$, i.e., if all cohomology generators have the same degree (see Lemma \ref{lem:lemma0} below).
\item If the model is the total space splits as a rational fibration whose base and fiber satisfy the Halperin conjecture, then it too satisfies the Halperin conjecture (see Markl \cite{Markl90}).
\end{itemize}

We use the following characterization of the Halperin conjecture due to Meier \cite{Mei82}. If the rational cohomology algebra $H^*(M;\Q)$ does not admit a derivation of degree $d < 0$, then the Halperin conjecture holds for $M$. Using this, the following is a well known and easy consequence:

\begin{lemma}\label{lem:lemma0}
Let $\delta$ be a derivation of negative degree on $H^*(M;\Q)$, where $M$ is an $F_0$ space. If $x \in H^i(M;\Q)$ for some $i > 0$ such that $\delta(x) \in H^0(M;\Q)$, then $\delta(x) = 0$. In particular, if $H^*(M;\Q)$ is generated by elements of the same degree, then $\delta = 0$.
\end{lemma}

In other words, any image of $\delta$ landing in $H^0(M;\Q)$ is zero.

\begin{proof}
Suppose $x^a$ is the maximal nonzero power of $x$. Since $\delta(x)$ is a non-zero element of $H^0(M;\Z)$, we have $\delta(x)x^a \neq 0$. However $\delta$ is a derivation, so it follows that $\delta(x^{a+1}) \neq 0$. This is a contradiction.

To prove the last statement, note that $\delta = 0$ if $\delta$ is zero on all generators. By the assumption on the degrees of the generators, these images either land in $H^0(M;\Q)$ or a zero group. Either way, these images are zero, so the proof is complete.
\end{proof}

The proof of the Halperin conjecture for Euler characterics up to $16$ now follows easily, so we prove it first.

\begin{theorem}\label{thm:HalperinEulerUpToEuler16}
The Halperin conjecture holds for $F_0$-spaces $M$ with $\chi(M)\leq 16$.
\end{theorem}
\begin{proof}
Recall from the previous section that the Euler characteristic of $M$ is given by the formula $\chi(M) = \prod \deg(\dif y_i)/\deg(x_i)$. Since each factor in the product is at least $2$, we have $\chi(M) \geq 2^k$ where $k$ is the number of cohomology generators, as above. By assumption, either $k \leq 3$ or $k = 4$ and $\deg(\dif y_i) = 2 \deg(x_i)$ for all $i$. In the first case, the conjecture holds by Lupton's result above. Suppose therefore that $k = 4$ and $\deg(\dif y_i) = 2 \deg(x_i)$ for all $i$.

Let $\delta$ be a degree $d$ derivation on $H^*(M;\Q)$ for some $d < 0$. If all of the $x_i$ have the same degree, then Lemma \ref{lem:lemma0} implies that $\delta = 0$. If this is not the case, then there exists some integer $l$ such that
	\[\deg(x_1) = \ldots = \deg(x_l) < \deg(x_{l+1}) \leq \ldots \leq \deg(x_4).\]
Since $\deg(\dif y_i) = 2\deg(x_i)$ for all $i$, it follows for degree reasons that $\dif y_i \in \Lambda(x_1,\ldots,x_l)$ for all $i \leq l$. In particular, the model splits as a rational fibration over the subalgebra $\Lambda(x_1,\ldots,x_l,y_1,\ldots,y_l)$. Since both the base and fiber have fewer than four generators, we conclude from Markl's result that the Halperin conjecture holds for $M$ as well.
\end{proof}

We proceed to the proof of the Halperin conjecture for $F_0$ spaces of dimension up to $16$. We require two more observations that will facilitate the arguments.

\begin{lemma}\label{lemH1}
Suppose $M$ is an $F_0$ space and that its pure model $(\Lambda V,\dif)$ is chosen as above. If $\delta$ is a derivation of negative degree on $H^*(M;\Q)$ such that $\delta(x_i)=0$ for $k-1$ of the $k$ generators $x_i$, then $\delta = 0$.
\end{lemma}

\begin{proof}
For this proof only, we reorder the $x_i$ so that $\delta(x_i)=0$ for all $i\geq 2$. We proceed by contradiction, i.e.~we assume that $\delta(x_1)\neq 0$. Then, by Poincar\'e duality, there is an element $x\in H^*(M;\Q)$ such that $\delta(x_1)x$ generates the top cohomology group. Moreover we may choose $x$ to be a monomial in the $x_i$. Write $x = x_1^l x'$ where $x'$ is a monomial in $x_2,\ldots,x_k$. It follows that $\delta(x_1^{l+1})x'$ generates top cohomology. But $x_1^{l+1}x' = 0$ for degree reasons and $\delta(x') = 0$ by assumption, so we compute that
	\[ 0= \delta(x_1^{l+1} x') - x_1^{l+1} \delta(x') = \delta(x_1^{l+1}) x' \neq 0,\]
a contradiction.
\end{proof}

\begin{lemma}\label{lemH2}
Suppose $M$ is an $F_0$ space and that its pure model $(\Lambda V,\dif)$ is chosen as above. If there exists $l<k$ such that $\deg(\dif y_l) < \deg(x_1) + \deg(x_{l+1})$, then $(\Lambda V,\dif)$ splits as a rational fibration with base algebra generated by $x_1, \ldots, x_l, y_1, \ldots, y_l$.
\end{lemma}

\begin{proof}
This lemma is proved by the observation that under the above assumptions, the regular sequence $\dif y_1, \ldots, \dif y_l$ lies in $\Lambda \langle x_1, \ldots, x_l\rangle$, so splitting follows.
\end{proof}

With these preparations, we proceed to the proof of the Halperin conjecture for $F_0$ spaces with dimension at most $16$. As a warm-up, we provide here a short proof in the case where there are at most three cohomology generators. Of course, this already follows by Lupton's theorem, but we include it here for the convenience of the reader.

\begin{lemma}\label{lem:Lupton16}
The Halperin conjecture holds for $F_0$ spaces $M$ such that $M$ has dimension at most $16$ and $H^*(M;\Q)$ has at most three generators.
\end{lemma}
\begin{proof}
If there are at most two generators of the cohomology algebra, the proof follows immediately from Lemmas \ref{lem:lemma0} and \ref{lemH1}. Suppose there are exactly three cohomology generators.

We denote by $(\Lambda\langle x_1,x_2,x_3,y_1,y_2,y_3\rangle,\dif)$ a corresponding minimal model with $\dif x_1=\dif x_2=\dif x_3=0$. Denote by $\delta$ a non-trivial derivation of negative degree on its cohomology algebra. Let the $x_i$ and the $y_i$ be ordered by degree. We want to show that it has to be trivial and assume the contrary. Denote the cohomology classes represented by $x_i$ by $[x_i]$. We make the following observations:
\begin{itemize}
\item $\delta([x_1]) = 0$ by Lemma \ref{lem:lemma0}.
\item $\delta([x_2]) \neq 0$ and $\delta([x_3]) \neq 0$ without loss of generality by Lemma \ref{lemH1}.
\item $\deg(x_1) < \deg(x_2)$ by degree reasons and Lemma \ref{lem:lemma0}.
\item $\deg(x_2) < \deg(x_3)$. Indeed, since $\delta$ maps $H^{\deg(x_2)}(M;\Q)$ linearly into a zero- or one-dimensional cohomology group, the property $\deg(x_2) = \deg(x_3)$ would imply that, up to a change of basis, $\delta([x_2]) = 0$, a contradiction.
\item $\deg(\dif y_1) \geq \deg(x_1) + \deg(x_2)$ without loss of generality by Lemma \ref{lemH2}. In fact, $\deg(\dif y_1) \geq \deg(x_1) + \deg(x_3)$ since otherwise $\dif y_1 \in \Lambda(x_1,x_2)$. If $\dif y_1 \in \Lambda(x_1)$, the model splits as a rational fibration. If, on the other hand, $\dif y_1 \not\in\Lambda(x_1)$, then applying $\delta$ to $[\dif y_1] = 0$ in cohomology yields a relation among $[x_1]$ and $[x_2]$, and this is a contradiction since $\dif y_1$ induces the relation of smallest degree.
\item $\deg(\dif y_2) \geq \deg(x_1) + \deg(x_3)$ without loss of generality by Lemma \ref{lemH2}.
\end{itemize}

Recall also that $\deg(\dif y_i) \geq 2 \deg(x_i)$ for all $1 \leq i \leq 3$. With these estimates in hand, the dimension formula,
	\[16 \geq n = \sum_{i=1}^{3} \of{\deg(\dif y_i) - \deg(x_i)},\]
implies that the tuple of degrees is $(2,4,6,7,7,11)$.

By the proof of the estimate $\deg(\dif y_1) \geq \deg(x_1) + \deg(x_3)$, we may assume that $\dif y_1$ contains a non-zero term involving $x_3$. By degree reasons, up to a scaling of $y_1$, we may assume $\dif y_1 = x_1 x_3 + p(x_1,x_2)$. By a similar argument, we may assume that $\dif y_2$ also has a non-zero $x_1 x_3$ term, and hence that $\dif y_2 = x_1x_3 + q(x_1,x_2)$ after scaling $y_2$. Finally, by replacing $y_1$ by $y_1 - y_2$, we have that $\dif y_1 \in \Lambda(x_1,x_2)$. Again, we obtain a contradiction if $\dif y_1$ has a non-zero term involving $x_2$, so we actually have $\dif y_1 \in \Lambda(x_1)$ and hence that the model of $M$ splits as a rational fibration.
\end{proof}

\begin{theorem}\label{thm:HalperinUpToDim16}
The Halperin conjecture holds for $F_0$ spaces of dimension up to $16$.
\end{theorem}

\begin{proof}
The proof proceeds by stepping through Tables \ref{tab6} to \ref{tab16}. Let $k$ denote the number of cohomology generators.

First, any model with $k \leq 3$ is covered by Lemma \ref{lem:Lupton16}. Second, any case where the first $k-1$ cohomology generators have the same degree is covered by Lemma \ref{lem:lemma0} and \ref{lemH1}. Third, any case where the last $k-1$ cohomology generators have the same degree also follows from these two lemmas. Indeed, if $\deg(x_1) < \deg(x_2) = \ldots = \deg(x_k)$, then $\delta$ maps $H^{\deg(x_2)}(M;\Q)$ linearly into a zero- or one-dimensional space. By a change of basis, we can arrange that $\delta([x_1]) = \ldots = \delta([x_{k-1}]) = 0$ and hence that $\delta = 0$. Finally, if some $\deg(\dif y_i) < \deg(x_1) + \deg(x_{i+1})$ as in Lemma \ref{lemH2}, then the model splits as a rational fibration where the base and fiber are $F_0$ spaces of smaller dimension. By induction over the dimension, this implies by Markl's theorem that $M$ satisfies the Halperin conjecture in this case.

Taking into account these observations, one can either scan the tables or apply the dimension formula to show that we only need to provide special arguments for the following five cases:
\begin{itemize}
\item $(2,2,4,4,3,5,7,7)$ and $\dim(M) = 14$.
\item $(2,2,4,4,3,7,7,7)$ and $\dim(M) = 16$.
\item $(2,2,4,4,5,5,7,7)$ and $\dim(M) = 16$.
\item $(2,2,4,6,3,5,7,11)$ and $\dim(M) = 16$.
\item $(2,2,2,4,4,3,3,5,7,7)$ and $\dim(M) = 16$.
\end{itemize}

We consider these cases one at a time, proving in each case either that there is no non-trivial derivation of negative degree or that the model for $M$ splits as a rational fibration and apply Markl's theorem.

\vspace{5mm}

$(2,2,4,4,3,5,7,7)$:
For degree reasons, we may write
	\begin{eqnarray*}
	\dif y_3 	&=& p(x_3,x_4) + q\\
	\dif y_4	&=& p'(x_3,x_4) + q'
	\end{eqnarray*}
where $p,p' \in \Lambda^2\langle x_3,x_4\rangle$ and where $q$ and $q'$ lie in the ideal $(x_1,x_2) \subseteq \Lambda\langle x_1,x_2,x_3,x_4\rangle$ generated by $x_1$ and $x_2$. Suppose for a moment that $p$ is a rational multiple of $p'$. Up to a change of basis, we may assume that $\dif y_3 = q \in (x_1,x_2)$. We derive a contradiction as follows. By the finite-dimensionality of $H^*(M;\Q)$, some power $[x_3]^m$ is zero. At the level of the model, we have that some $\alpha_i \in \Lambda\langle x_1,x_2,x_3,x_4\rangle$ exist such that
	\[x_3^m = \dif\of{\sum_{i=1}^4 \alpha_i y_i} = \sum_{i=1}^4 \alpha_i \dif y_i = \sum_{i=1}^3 \alpha_i \dif y_i + \alpha_4(p' + q').\] Since $\dif y_i \in (x_1,x_2)$ for $i \leq 3$ and since $q' \in (x_1,x_2)$, it follows that $p'$ is a rational multiple of $x_3^2$. On the other hand, we can apply the same argument using a sufficiently large power of $x_4$ to conclude that $p'$ is a rational multiple of $x_4^2$. This is a contradiction, hence we may assume that $p$ and $p'$ are not multiples of each other.

Let $\delta$ be a non-trivial derivation on $H^*(M;\Q)$ of negative degree. By Lemma \ref{lem:lemma0}, we may assume that $\delta([x_1]) = \delta([x_2]) = 0$ and that $\delta$ has degree $-2$. Moreover, in view of Lemma \ref{lemH2}, we may assume that $\delta([x_3])$ and $\delta([x_4])$ are linearly independent, since otherwise a change of basis would yield $\delta([x_3]) = 0$ and hence that $\delta = 0$. By a further change of basis, we may assume that  $\delta([x_3])=[x_1]$ and $\delta([x_4]) = [x_2]$.

Note, in particular, that $\delta$ maps $[q], [q'] \in H^*(M;\Q)$ into the subalgebra generated by $[x_1]$ and $[x_2]$. Hence $\delta^2$ maps $[q]$ and $[q']$ to zero in $H^4(M;\Q)$. Moreover, if we write
	\[p = a x_3^2 + b x_3 x_4 + c x_4^2,\]
for $a,b,c\in \qq$, then we have that
	\[\delta^2([p]) = 2a [x_1]^2 + 2b [x_1][x_2] + 2c [x_2]^2,\]
and similarly for $[p']$. Since the triple $(a,b,c)$ and the corresponding triple for $p'$ are not multiples of each other, we have constructed two independent relations in degree four. This contradicts the fact that there is only one homotopy generator in degree three, so the proof is complete.

\vspace{5mm}

$(2,2,4,4,3,7,7,7)$:
As in the previous case, we may assume that $\delta([x_3])=[x_1]$, $\delta([x_4])=[x_2]$, and $\delta([x_i]) = 0$ for $i \in \{1,2\}$. Moreover, it follows as in the last case that there is some pair $y_i$ and $y_j$ such that
$\dif y_i$ and $\dif y_j$ are linearly independent modulo the ideal $(x_1,x_2)$. As in the previous case, this implies the existence of two linearly independent relations in degree four and hence provides a contradiction to degree considerations.

\vspace{5mm}

$(2,2,4,4,5,5,7,7)$:
As in the previous case, we may assume that $\delta([x_3])=[x_1]$, $\delta([x_4])=[x_2]$, and $\delta([x_i]) = 0$ for $i \in \{1,2\}$. The proof can be handled similarly to the previous two cases, but there is a shortcut here. Without restriction, we may assume that that $\dif y_3$ has $x_3^2$ and possibly $x_4^2$ as non-trivial summands, up to multiples. We compute that $\delta^2(\dif y_3)$ has summands $x_1^2$ and possibly $x_2^2$, up to non-trivial multiples. Since there is no relation in degree $3$, we have a contradiction.

\vspace{5mm}

$(2,2,4,6,3,5,7,11)$:
As in the previous cases, we may assume $\delta([x_1]) = \delta([x_2]) = 0$ and that $\delta$ has degree two. In view of Lemma \ref{lemH1}, we may assume further that $\delta([x_3]) \in \langle [x_1], [x_2]\rangle$ and $\delta([x_4]) \in \langle [x_3],[x_1]^2, [x_1][x_2],[x_2]^2\rangle$ are non-zero. Up to an isomorphism of the minimal model, we may assume that either $\delta([x_4])=[x_3]$ or $\delta([x_4])\in \langle [x_1]^2,[x_1][x_2],[x_2]^2\rangle$. Further, we may assume $\delta([x_3]) = [x_1]$.

Using the finite-dimensionality of $H^*(M;\Q)$, it follows that $x_4^2$ appears as a non-trivial summand of $\dif y_4$. Up to scaling $x_4$, we may express
	\[\dif y_4 = x_4^2 + k_1 x_3^3 + l(x_1,x_2) x_3 x_4 + p,\]
where $k_1 \in \Q$, $l(x_1,x_2)$ is a linear function of $x_1$ and $x_2$, and $p \in \ker(\delta^3)$. We break the proof into two cases.

First, we suppose that $\delta([x_4])=[x_3]$. For this we apply $\delta^4$ to the relation in cohomology induced by $ y_4$. Since $\delta^4([x_3]^3) = 0$, we conclude $0 = \delta^4([x_4]^2) = 6 [x_1]^2$. This implies that $\dif y_1$ is a multiple of $x_1^2$ and hence that the model splits as a rational fibration over $\Lambda(x_1,y_1)$. By Markl's result, the Halperin conjecture follows.

Second, we suppose that $\delta([x_4]) \in \langle [x_1]^2,[x_1][x_2],[x_2]^2\rangle$. For this case, we apply $\delta^3$ to the relation induced by $\dif y_4$. This time, $\delta^3([x_4]^2) = 0$, so we conclude that $0 = \delta^3(k_1[x_3]^3) = 6k_1[x_1]^3$. We consider now two subcases, according to whether $k_1=0$ or $[x_1]^3=0$.

First suppose $k_1 = 0$. Using the finite-dimensionality of $H^*(M;\Q)$, the fact that $x_3^3$ does not appear in $\dif y_4$ implies that $x_3^2$ must appear as a summand in $\dif y_3$. Up to scaling $y_3$, we may write
	\[\dif y_3 = x_3^2 + l'(x_1,x_2) x_4 + p',\]
where $l'(x_1,x_2)$ is a linear function of $x_1$ and $x_2$ and $p' \in \ker(\delta^2)$. In fact, $\delta^2$ also sends $[x_4]$ to zero since we are assuming $\delta([x_4])$ is a polynomial in $[x_1]$ and $[x_2]$. Hence applying $\delta^2$ to the induced relation in cohomology yields $0 = 2 [x_1]^2$. As before, this implies a rational fibration splitting, and hence that the Halperin conjecture holds.

Finally, suppose that $[x_1]^3 = 0$. This implies that
	\[x_1^3 = m(x_1,x_2) \dif y_1 + c \dif y_2\]
for some linear function $m(x_1,x_2)$ of $x_1$ and $x_2$ and some $c \in \Q$. Now, for degree considerations, we may express
	\[\dif y_2 = l''(x_1,x_2) x_3 + p'',\]
where $l''(x_1,x_2)$ is a linear function of $x_1$ and $x_2$ and $p'' \in \Lambda\langle x_1,x_2\rangle \subseteq \ker(\delta)$. Moreover, unless the model of $M$ splits as a rational fibration over the subalgebra $(\Lambda\langle x_1,x_2,y_1,y_2\rangle,\dif)$, we may assume that $\l''(x_1,x_2) \neq 0$. Applying $\delta$ to the induced relation in cohomology yields $0 = l''([x_1], [x_2]) [x_1]$. Since the only relation in this degree is that induced by $y_1$, we conclude from this that $\dif y_1 \in \Lambda\langle x_1,x_2\rangle$ is divisible by $x_1$. But now the relation above for $x_1^3$ implies that $\dif y_2$ is also divisible by $x_1$. Looking again at the expression for $\dif y_2$, we conclude that $l''(x_1,x_2)$ is a (non-trivial) multiple of $x_1$. Since $l''([x_1],[x_2])[x_1] = 0$, we conclude that $[x_1]^2 = 0$ and hence that, again, the model for $M$ splits as a rational fibration over $\Lambda \langle x_1,y_1\rangle$. This proves the Halperin conjecture for this tuple of homotopy generator degrees.

\vspace{5mm}

$(2,2,2,4,4,3,3,5,7,7)$: Here we may assume $\delta([x_i]) = 0$ for $i \in \{1,2,3\}$ and that $\delta([x_4]) = [x_1]$ and $\delta([x_5]) = [x_2]$. The proof here is similar to the first few exceptional cases, where the ideal $(x_1,x_2)$ is replaced by the ideal $(x_1,x_2,x_3)$.

\vspace{5mm}

This concludes the proof of Halperin's conjecture in the five special cases above. Altogether, this completes the proof of the  conjecture in dimensions up to $16$.
\end{proof}

\smallskip\section{Positive curvature and rational ellipticity}\label{sec:Elliptic}\smallskip

We now combine the information of the previous sections in order to classify rationally elliptic Riemannian manifolds of dimension at most $16$ with positive sectional curvature and torus symmetry. The additional assumption of rational ellipticity does not add to our understanding in dimensions two, four, and six since the existing theorems do not see improvement upon adding the assumption of rational ellipticity. Indeed, each of these results show that the manifold is rational elliptic. Starting in dimension eight, however, Dessai's result is one where the conclusion (an Euler characteristic, signature, and elliptic genus calculation) is improved to a rational homotopy classification by adding the assumption of rational ellipticity (see \cite[Theorem 1.2]{Dessai11}). We have corresponding rational homotopy classifications in dimensions $10$, $14$, and $16$, and a partial result along these lines in dimension $12$.

\begin{theorem}\label{thm:dim10elliptic}
If $M^{10}$ is a closed, simply connected, rationally elliptic Riemannian manifold with positive curvature and $T^3$ symmetry, then $M$ is rationally homotopy equivalent to $\s^{10}$, $\C\pp^5$, or $\s^2 \times \HH\pp^2$.
\end{theorem}

\begin{proof}
We may assume that $M$ is not homeomorphic to $\s^{10}$. By our classification in dimension $10$ (Theorem \ref{thm:dim10}), $\chi(M) = 6$, $H_2(M;\Q) = \Q$, and $H^{10}(M;\Q)$ is generated by a product of elements of degree two and four. Rational ellipticity implies $H^{2i+1}(M;\Q) = 0$ for all $i$, and it follows that $H_*(M;\Q) \cong H_*(\C\pp^5;\Q)$. Fix generators $z \in H^2(M;\Q)$ and $x \in H^4(M;\Q)$. 

By Theorem \ref{thm:dim10}, in the case when cohomology is generated by $x$ and $z$, the manifold $M$ has the rational type of a complex projective space (corresponding to $m\neq 0$) or $\s^2\times \hh\pp^2$ (when $m=0$) using intrinsic formality of these spaces.

In the remaining case, the element $x^3$ does not vanish and generates $H^6(M;\qq)$. Hence, by Poincar\'e duality, $M\simeq_\qq \cc\pp^5$.
\end{proof}

\begin{theorem}\label{thm:dim12elliptic}
Let $M^{12}$ be a closed, simply connected, rationally elliptic Riemannian manifold with positive curvature and $T^3$ symmetry. If $\chi(M) \in \{2,4,7\}$, then one of the following occurs:
\begin{itemize}
\item If $\chi(M)=2$, then $M\simeq_\qq \s^{12}$.
\item If $\chi(M)=4$, then $M \simeq_\qq \hh\pp^3$, $M \simeq_\qq \s^2\times \s^{10}$, $M \simeq_\qq \s^4\times \s^8$, or the rational type of $M$ comes out of a $1$-parameter family tensoring to the real homotopy type of $\s^6\times \s^6$.
\item If $\chi(M)=7$, then  $M\simeq_\qq \cc\pp^6$.
\end{itemize}
If $\chi(M) \not\in \{2,4,7\}$, then $\chi(M) \in \{6,8,9,10,12\}$ and correspondingly $M$ has the Betti numbers and homotopy Betti numbers of one of the following spaces:
\begin{itemize}
\item $\s^8\times \cc\pp^2$ or $\s^4\times \hh\pp^2$.
\item $\s^6\times \cc\pp^3$, $\s^2\times \s^2\times \s^8$, $\s^2\times \s^4\times \s^6$, or $\s^4\times\s^4\times\s^4$.
\item $\cc\pp^2\times \hh\pp^2$.
\item $\s^4 \times \cc\pp^4$.
\item $\s^2\times \cc\pp^5$, $\s^2\times \s^6 \times \cc\pp^2$, $\cc\pp^2\times \s^4\times \s^4$, or $\s^2\times\s^2 \times\hh\pp^2$.
\end{itemize}

\end{theorem}

\begin{remark}
Theorem \ref{thm:dim12} further computes the signature in this setting, but  this additional information does not exclude any of the manifolds listed in the conclusion of this theorem. For example, in the case of rational homotopy generators in degrees $2$, $6$, $7$, and $11$, we construct one class of possible minimal models by
\begin{align*}
(\Lambda \langle x,y, x',y'\rangle,\dif)
\end{align*}
with $\deg x=2, \deg y=6, \deg x'=7, \deg y'=11$, $\dif x=\dif y=0$, $\dif x'= xy$,  $\dif y'=x^6+y^2$. Its intersection form is represented by
\begin{align*}
\begin{pmatrix}
1 & 0\\
0 & -1
\end{pmatrix}
\end{align*}

For the configuration $(2,4,6,3,7,11)$ we may use $\s^2\times \s^4 \times \s^6$ with the intersection matrix being the same with respect to the basis $\frac{1}{2}([\s^6]+[\s^2][\s^4]), \frac{1}{2}([\s^6]-[\s^2][\s^4])$.

Similarly, when $\chi(M)\in \{10,12\}$, we obtain spaces which can be realized as manifolds by a product containing a sphere factor, i.e.~by boundaries. Since the signature is a bordism invariant, these manifolds have vanishing signatures.
\end{remark}

\begin{proof}
If $\chi(M) = 2$ or $\chi(M) = 7$, Table \ref{tab12} implies that $M$ has the rational homotopy groups of $\s^{12}$ or $\C\pp^6$. By formality, the rational homotopy classification follows in this case.

Suppose next that $\chi(M) = 4$. Table \ref{tab12} implies that $M$ has the rational homotopy groups of one of four types. We consider these cases one at a time. Listed first in each case is the sequence of degrees of homotopy generators.
	\begin{itemize}
	\item $(4,15)$. In this case, $M$ has the rational homotopy groups of $\HH\pp^3$ and hence the same rational homotopy type.
	\item $(2,10,3,19)$. In this case, the generator in degree three must kill the square of the generator in degree two, so it is straightforward to see in this case that $M$ has the rational homotopy type of $\s^2 \times \s^{10}$.
	\item $(4,8,7,15)$. This case is similar to the previous case and yields the rational homotopy type of $\s^4 \times \s^8$.
	\item $(6,6,11,11)$.  In this case, $M$ has a minimal model of the form $(\Lambda \langle u,v,u',v'\rangle,\dif)$ with $\deg u=\deg v=6$, $\deg u'=\deg v'=11$, $\dif u=\dif v=0$. Since the signature vanishes, we obtain that the intersection form over the reals is represented by the matrix
	\begin{align*}
	\begin{pmatrix}
	1 & 0\\
	0& -1
	\end{pmatrix}
	\end{align*}
This relates to the minimal model over $\rr$ by determining the differentials $\dif u'=u^2+v^2$ and $\dif v'=uv$ up to isomorphism. Let us now investigate how many rational homotopy types tensor to this real homotopy type. Since the product of the two elements $u,v$ is exact over $\rr$ if and only if it is exact over $\qq$, we see that all the rational types of this space are given by
\begin{align*}
&(\Lambda \langle u,v,u',v'\rangle,\dif),
\\ &\deg u=\deg v=6, \deg u'=\deg v'=11, \dif u=\dif v=0
\\ & \dif u'=u^2+k v^2, \dif v'=uv, k\in \qq
\end{align*}
and $k>0$ may be chosen not to be a square number over $\qq$. It is easy to see that for $k=1$, this is the model for $\s^6\times \s^6$.
	\end{itemize}

We proceed to the proof in the second case, where $\chi(M) \not\in \{2,4,7\}$. Here our Euler characteristic calculation (Theorem \ref{thm:dim12}) implies that $\chi(M) \in \{6, 8, 9, 10, 12\}$. In particular, $M$ is an $F_0$ space and hence admits a pure model. Recall from \cite[p.~446]{FelixHalperinThomas01} that the Betti numbers of a pure space only depend on the rational homotopy groups, i.e.~pure spaces with the same rational homotopy groups have the same Betti numbers. As a result, we only need to compute the dimensions of the rational homotopy groups. This information can be immediately deduced from Table \ref{tab12}.

\end{proof}

Next we come to the rational homotopy classifications in dimensions $14$ and $16$.

\begin{theorem}\label{thm:dim14elliptic}
If $M^{14}$ is a closed, simply connected, rationally elliptic Riemannian manifold with positive curvature and $T^4$ symmetry, then $M$ is rationally homotopy equivalent to $\s^{14}$, $\C\pp^7$, or $\s^2 \times \HH\pp^3$.
\end{theorem}

\begin{proof}
This follows directly from Theorem \ref{thm:dim14}, since $\chi(M)=2$ for a rationally elliptic space $M^n$ implies $M^n\simeq_\qq \s^n$ and since the rational homotopy types of $\C\pp^7$ and $\s^2 \times \HH\pp^3$ (corresponding to $m \neq 0$ and $m = 0$ in the conclusion of Theorem \ref{thm:dim14}, respectively) are intrinsically formal.
\end{proof}

\begin{theorem}\label{thm:dim16elliptic}
If $M^{16}$ is a closed, simply connected, rationally elliptic Riemannian manifold with positive curvature and $T^4$ symmetry, then $M$ is rationally homotopy equivalent to a compact, rank one symmetric space, i.e., to $\s^{10}$, $\C\pp^5$, $\HH\pp^4$, or  $\Ca\pp^2$.
\end{theorem}

\begin{proof}
From Theorem \ref{thm:dim16} we derive that $\chi(M)\in \{2,3,5,9\}$. In the first three cases, Table \ref{tab16} implies that $M$ has singly generated rational cohomology with generator in degree $16$, $8$, or $4$, and so the result follows by formality. Suppose $\chi(M) = 9$. By Theorem \ref{thm:dim16} again, the fifth power of a generator $x \in H^2(M;\Q)$ is non-zero. Since $M$ is an $F_0$ space, the odd Betti numbers are zero, and by this fact the even Betti numbers are equal to one. By Poincar\'e duality, $M$ is a rational cohomology $\C\pp^8$ and hence a rational homotopy $\C\pp^8$ by formality.
\end{proof}

In \cite{AmannKennard15}, the authors proved the Wilhelm conjecture for spaces with rational cohomology algebra generated by one element. We have the following consequence.

\begin{corollary}
The Wilhelm conjecture holds for closed, simply connected, rationally elliptic manifolds of dimension $16$ that admit a Riemannian metric with positive sectional curvature and $T^4$ symmetry.
\end{corollary}

\smallskip\section{Biquotients}\label{sec:biquotients}\smallskip

In this section, we combine our Euler characteristic and other calcultions with work of Kapovitch--Ziller, Totaro, and DeVito to provide diffeomorphism classifications for biquotients in dimensions $10$, $14$, and $16$ that admit positively curved metrics with torus symmetry. In particular, we obtain a characterization of the Cayley plane in this context. We also consider dimension $12$ and restrict further the case of symmetric spaces to obtain a partial diffeomorphism classification.

The main additional ingredient is the classification of biquotients whose rational cohomology is four--periodic in the sense of \cite[Definition 1.1]{Kennard13}. This class includes those with singly generated cohomology, and this case was classified in Kapovitch and Ziller \cite{KapovitchZiller04}. In even dimensions, there is one more possible rational cohomology ring, namely, that of $\s^2 \times \HH\pp^m$. Here we require recent work of DeVito \cite{DeVito-pre2}.

To be clear, we state explicitly that in the subsequent theorem the biquotient structure does not have to be related to the positively curved metric.

\begin{theorem}\label{thm:dim101416biquotient}
Let $M^n$ be a closed, simply connected biquotient that independently admits a positively curved Riemannian metric with $T^r$ symmetry.
	\begin{itemize}
	\item If $n = 10$ and $r \geq 3$, then $M$ is diffeomorphic to $\s^{10}$, $\C\pp^5$, $\s^2 \tilde{\times}\HH\pp^2$, $N_1^{10}$, or $N_2^{10}$.
	\item If $n = 14$ and $r \geq 4$, then $M$ is diffeomorphic to $\s^{14}$, $\C\pp^7$, $\s^2 \tilde{\times}\HH\pp^3$, $N_1^{14}$, or $N_2^{14}$.
	\item If $n = 16$ and $r \geq 4$, then $M$ is diffeomorphic to $\s^{16}$, $\C\pp^8$, $\HH\pp^4$, or $\Ca\pp^2$.
	\end{itemize}
Here $\s^2 \tilde \times \HH\pp^m$ denotes one of the two diffeomorphism types of total spaces of $\HH\pp^m$--bundles over $\s^2$ whose structure group reduces to a circle acting linearly, and $N_1^{4k-2}$ and $N_2^{4k-2}$ denote the free circle quotients $\SO(2k+1)/(\SO(2k-1)\times \SO(2))$ and $\Delta\SO(2)\backslash\SO(2k+1)/\SO(2k-1)$ of the unit tangent bundle of $\s^{2k}$ described in \cite{KapovitchZiller04}.
\end{theorem}

\begin{proof}
First suppose $M$ has the rational homotopy type of $\s^2\times \hh\pp^2$ or $\s^2\times \hh\pp^3$. From DeVito's classification we deduce that $M^{10}$ is diffeomorphic to one of two bundles over $\s^2$ or with fiber $\hh\pp^2$ or $[G_2/\SO(4)]\times \s^2$. For $M^{14}$ we obtain that $M$ is one of the two respective $\hh\pp^3$-bundles over $\s^2$. Now, in dimension ten, we know that $H_2(M;\zz)=\zz$, whereas $H_2([G_2/\SO(4)]\times \s^2)=\zz_2\oplus \zz$.

Next suppose that the rational homotopy type is not as above. The calculations above for rationally elliptic spaces imply that $M$ has the rational homotopy type of a compact, rank one symmetric space (see Theorems \ref{thm:dim10elliptic}, \ref{thm:dim14elliptic}, and \ref{thm:dim16elliptic}). In particular, the rational cohomology of $M$ is generated by one element, so we can apply the classification of Kapovitch and Ziller (see \cite[Theorem 0.1]{KapovitchZiller04}). In these dimensions, either the conclusion of Theorem \ref{thm:dim101416biquotient} holds, or $M$ is diffeomorphic to one of the two rational complex projective spaces in dimensions $10$ or $14$. These spaces are $N_1$ and $N_2$ from the conclusion of the theorem.
\end{proof}

A similarly strong result in dimension $12$ seems out of reach. We make a few general remarks on the structure of a $12$--dimensional biquotient admitting positive sectional curvature and $T^3$ symmetry. We then restrict further to the case of symmetric spaces and provide a diffeomorphism classification in this case.

Suppose that $M^{12}$ is a closed, simply connected Riemannian manifold with positive sectional curvature and $T^3$ symmetry. Assume moreover that $M$ admits a possibly independent biquotient structure, i.e., suppose that $M$ is diffeomorphic to a quotient $\biq{G}{H}$ of $G$ by a two-sided action by a subgroup $H \subseteq G$. By the results of Kapovitch--Ziller \cite{KapovitchZiller04} or Totaro \cite{Totaro02}, we may replace $G$ and $H$, if necessary, so that $G$ is a connected, simply connected Lie group and $H$ is a product of a torus $T^r$ with a connected, simply connected Lie group. In other words, $G = \prod_{i=1}^k G_i$ and $H = T^r \times \prod_{i=1}^l H_i$, where $G_i$ and $H_i$ are compact, one--connected, simple Lie groups. Since $\chi(M) > 0$ under our assumptions, the ranks of $G$ and $H$ agree, so $l = k - r$. Moreover, an extension of the arguments in \cite[Section 1]{KapovitchZiller04} or \cite[Lemma 3.3 and Corollary 4.6]{Totaro02} show that $G$ has at most three simple factors, i.e., $k \leq 3$. Next, we require Totaro's work on the ``contribution'' of each simple group $G_i$ to the odd rational homotopy groups of $M$. This work implies that, if $G_i$ has the rational homotopy type of $\prod_j \s^{2n_j - 1}$, then the largest value of $n_j$ appearing is at most $2n = 24$. In particular, there exist finitely many such $G_i$. Finally we recall that the quotient map $G \to M$ is a fibration with fiber $H$, so the associated long exact homotopy sequence relates the rational homotopy groups of $G$ and $H$ (which are both concentrated in odd degrees) with those of $M$. In particular, this implies that $r = b_2(M) = \dim \pi_2(M) \tensor \Q \leq 2$ by Table \ref{tab12}. Putting all of this together, we have the following structure theorem.

\begin{theorem}\label{thm:dim12biquotients}
Let $M^{12}$ be a closed, simply connected Riemannian manifold with positive sectional curvature and $T^3$ symmetry. Assume $M$ admits a possibly independent biquotient structure. Then there exist $k \in \{1,2,3\}$ and $r \in \{1,2\}$ and compact, one--connected, simple Lie groups $G_i$ and $H_i$ such that $\pi_j(G_i)\tensor \Q = \pi_j(H_i) \tensor \Q = 0$ for $j \geq 48$ and such that $M$ is diffeomorphic to a biquotient of the form $\biq G H$, where $G = \prod_{i=1}^k G_i$ and $H = T^r \times \prod_{i=1}^{k-r} H_i$.
\end{theorem}

Using the last statement together with the classification in Table \ref{tab12} of the rational homotopy groups of $M$ (with dimension $12$ and Euler characteristic $2$, $4$, $6$, $7$, $8$, $9$, $10$, or $12$), one could analyze the possible groups $G$ and $H$ appearing in the conclusion of this theorem. Indeed, the theorem together with the classification of compact, simply connected, simple Lie groups implies that there are finitely many choices of $G$. Given $G$ and $M$, the rational homotopy groups of $H$ follow from the long exact homotopy sequence associated to the fibration $G \to M$, so there are only finitely many possibilities for $H$. We could therefore enumerate the finite number of possibilities for $G$ and $H$, however we do not pursue this here.

Instead, we restrict ourselves further to the class of symmetric spaces. These are homogeneous spaces and therefore biquotients, so everything above carries over. However, we proceed by a different route to get a classification in this case.

\begin{theorem}\label{thm:dim12symmetricspaces}
Let $M^{12}$ be a compact, simply connected symmetric space. Assume $M$ admits a possibly independent Riemannian metric with positive sectional curvature and $T^3$ symmetry. One of the following occurs:
	\begin{itemize}
	\item $\chi(M) = 2$, $M$ is spin, and $M$ is $\s^{12}$.
	\item $\chi(M) = 4$, $M$ is spin, and $M$ is $\HH\pp^3$ or a product of the form $\s^{2k} \times \s^{12-2k}$.
	\item $\chi(M) = 6$, $M$ is spin, and $M$ is $\s^4 \times \HH\pp^2$ or $\s^4 \times \of{\Gtwo/\SO(4)}$.
	\item $\chi(M) = 6$, $M$ is not spin, and $M$ is $\SO(7)/\SO(3)\times\SO(4)$ or $\s^8 \times \C\pp^2$.
	\item $\chi(M) = 7$, $M$ is not spin, and $M$ is $\C\pp^6$.
	\item $\chi(M) = 8$, $M$ is spin, and $M$ is $\SO(8)/\SO(2) \times \SO(6)$, $\Sp(3)/\Un(3)$, $\s^6 \times \C\pp^3$, or a product of the form $\s^{2k} \times \s^{2l} \times \s^{12-2k-2l}$.
	\item $\chi(M) = 8$, $M$ is not spin, and $M$ is $\s^6 \times \of{\SO(5) / \SO(2) \times \SO(3)}$.
	\item $\chi(M) = 9$, $M$ is not spin, and $M$ is $\C\pp^2 \times \HH\pp^2$ or $\C\pp^2 \times \of{\Gtwo/\SO(4)}$.
	\item $\chi(M) = 10$, $M$ is not spin, and $M$ is $\s^4 \times \C\pp^4$.
	\item $\chi(M) = 12$, $M$ is not spin, and $M$ is $\s^2 \times \of{\SO(7)/\SO(2) \times \SO(5)}$ or a product of the form $\s^{2k} \times \s^{8-2k} \times \C\pp^2$.
	\end{itemize}
\end{theorem}
We note that the signature in each case is $0$ or $\pm 1$, according to the parity of $M$. In addition, the elliptic genus is constant for all of these manifolds since they are homogeneous (see Hirzebruch--Slodowy \cite{HS90}). These remarks are consistent with the calculation in Theorem \ref{thm:dim12} of the Euler characteristic, signature, and elliptic genus.

We also remark that Theorem \ref{thm:dim12symmetricspaces} could be improved if the constant $C(6)$ defined in the introduction is shown to be six. Indeed, if $C(6) = 6$, our Euler characteristic calculation in dimension $12$ implies that $\chi(M) \neq 12$, $\chi(M) \neq 10$, and that $\chi(M) = 8$ only if $M$ is not spin.

\begin{proof}
Let $M^{12}$ be a closed, simply connected symmetric space that admits a (possibly independent) Riemannian metric with positive curvature and $T^3$ symmetry. By Theorem \ref{thm:dim12}, either $\chi(M) \in \{2,4,6,8\}$ or $M$ is not spin and $\chi(M) \in \{7,9,10,12\}$. Note in the first case, the theorem does not claim that $M$ is spin, however we will see for symmetric spaces, $\chi(M) \in \{2,4\}$ only if $M$ is spin.

By the classification of symmetric spaces, $M = \prod_{i=1}^k M_i$ where each $M_i$ is a compact, simply connected, irreducible symmetric space. Since $\dim(M) = 12$, there are a finite number of possibilities to check. In fact, the classification also implies that each $M_i$ has Euler characteristic zero or Euler characteristic at least two. Since $\chi(M) > 0$, we conclude that each $\chi(M_i) \geq 2$ (in particular, no Lie group factors appear) and that the number of factors is at most three.

Consider first the case where $M = M_1$, i.e., where $M$ is irreducible. For this case, we simply enumerate all compact, simply connected, irreducible symmetric spaces of dimension $12$. We first eliminate those spaces with Euler characteristic zero (i.e., any space $G/H$ where $G$ and $H$ have unequal rank). Second, we eliminate the rank two complex Grassmannian $\Un(5)/\Un(2)\times\Un(3)$, which has Euler characteristic $10$ and is not spin, because it has signature $\pm 2$ (see Shanahan \cite[Section 5]{Shanahan79}). Third, we record each space's signature and whether the spaces are spin, and we check that this data is consistent with our obstructions mentioned above (Theorem \ref{thm:dim12}). For determining which irreducible symmetric spaces are spin, we rely on Cahen--Gutt \cite{CahenGutt88} (cf. \cite[Table 4]{GadeaGonzalez-DavilaOubina-pre}). For the signature, we note that the signature vanishes for any simply connected, spin homogeneous space of dimension not divisible by eight (see Hirzebruch--Slodowy \cite[Section 2.3]{HS90}). For the signature of the oriented Grassmannians, see Slodowy \cite{Slodowy92}. Finally, we exclude $\SO(8)/\Un(4)$ from the table, since it is diffeomorphic to $\SO(8)/\SO(2)\times\SO(6)$ by way of low-dimensional, accidental isomorphisms (see Ziller \cite[Section 6.3]{Ziller-symspsnotes}). The result is Table \ref{tab:dim12sym-irreducible}.

\begin{table}[h]
\centering \caption{Irreducible symmetric spaces $M^{12}$ with $2 \leq \chi(M) \leq 12$ and $|\sigma(M)| \leq 1$.}
\label{tab:dim12sym-irreducible}
\begin{tabular}{ c | c | c | c }
$M$ 								&$\chi(M)$	& spin?	& $|\sigma(M)|$ \\\hline
$\s^{12}$								& $2$		& yes	& $0$	\\
$\HH\pp^3$							& $4$		& yes	& $0$	\\
$\SO(7)/\SO(3)\times\SO(4)$				& $6$		& no 		& $0$	\\ 
$\C\pp^6$								& $7$		& no		& $1$	\\
$\SO(8)/\SO(2)\times\SO(6)$, $\Sp(3)/\Un(3)$	& $8$		& yes 	& $0$	\\
\end{tabular}
\end{table}

Next consider the case where $M = M_1 \times M_2$ or $M = M_1 \times M_2 \times M_3$ where each $M_i$ is irreducible. In these cases, $2 \leq \dim(M_i) \leq 10$ and $2 \leq \chi(M_i)\leq 6$ by the multiplicativity of the Euler characteristic and the fact that $\chi(M) \leq 12$. Using Ziller \cite[Section 6.3]{Ziller-symspsnotes} again, we eliminate redundancy in our list by observing the following accidental isomorphisms:
	\begin{enumerate}
	\item $\SO(4)/\SO(2) \times \SO(2)$ is diffeomorphic to $\s^2 \times \s^2$, and hence it is not irreducible.
	\item $\Sp(2)/\Un(2)$ is diffeomorphic to $\C\pp^3$.
	\item $\SO(6)/\Un(3)$ is diffeomorphic to $\C\pp^3$.
	\end{enumerate}
We also record which spaces are spin using the classification of Cahen and Gutt mentioned above. The list of possible irreducible factors is presented in Table \ref{tab:dim12sym-reducible}.

\begin{table}[h]
\centering \caption{Irreducible symmetric spaces $M^n$ with $2 \leq n \leq 10$ and $2 \leq \chi(M) \leq 6$.}
\label{tab:dim12sym-reducible}
\begin{tabular}{ c | c | c | c }
$M$ 							& $\dim(M)$	& $\chi(M)$ 	& spin? 	\\\hline
$\s^{n}$							& $n$		& $2$		& yes	\\
$\C\pp^2$							& $4$		& $3$		& no		\\
$\HH\pp^2$, $\Gtwo/\SO(4)$						& $8$		& $3$		& yes	\\
$\C\pp^3$							& $6$		& $4$		& yes	\\
$\SO(5)/\SO(2)\times\SO(3)$			& $6$		& $4$		& no		\\
$\C\pp^4$							& $8$		& $5$		& no		\\
$\SO(6)/\SO(2) \times \SO(4)$, $\Un(4)/\Un(2) \times \Un(2)$		& $8$		& $6$		& yes	\\
$\SO(7)/\SO(2) \times \SO(5)$			& $10$		& $6$		& no		\\
$\C\pp^5$							& $10$		& $6$		& yes	
\end{tabular}
\end{table}

With Table \ref{tab:dim12sym-reducible} complete, it is not difficult to enumerate the possible products $M_1 \times M_2$ and $M_1 \times M_2 \times M_3$ that have dimension $12$ and Euler characteristic satisfying $\chi(M) \in \{2,4,6,8\}$ or $\chi(M) \in \{7,9,10,12\}$, where the latter case only occurs if $M$ is not spin. Here it is useful to recall that the product of manifolds is spin if and only if each factor is spin. For example, we can exclude the products of the form $\s^2 \times N$ where $N$ is one of the three spin manifolds in Table \ref{tab:dim12sym-reducible} with Euler characteristic six.
\end{proof}


\vfill
\begin{center}
\noindent
\begin{minipage}{\linewidth}
\small \noindent \textsc
{Manuel Amann} \\
\textsc{Fakult\"at f\"ur Mathematik}\\
\textsc{Institut f\"ur Algebra und Geometrie}\\
\textsc{Karlsruher Institut f\"ur Technologie}\\
\textsc{Englerstra\ss{}e 2, Karlsruhe}\\
\textsc{76131 Karlsruhe}\\
\textsc{Germany}\\
[1ex]
\textsf{manuel.amann@kit.edu}\\
\url{http://topology.math.kit.edu/21_54.php}
\end{minipage}
\end{center}

\vspace{10mm}

\begin{center}
\noindent
\begin{minipage}{\linewidth}
\small \noindent \textsc
{Lee Kennard} \\
\textsc{Department of Mathematics}\\
\textsc{University of Oklahoma}\\
\textsc{Norman, OK 73019-3103}\\
\textsc{USA}\\
[1ex]
\textsf{kennard@math.ou.edu}\\
\url{http://math.ou.edu/~kennard}
\end{minipage}
\end{center}

\end{document}